\documentclass[11pt]{amsart}
\usepackage[verbose,letterpaper,tmargin=3cm,bmargin=3cm,lmargin=2.5cm,rmargin=2.5cm]{geometry}
\pdfoutput=1 

\usepackage{amsmath,amssymb,amsthm,mathrsfs,latexsym, mathtools}
\usepackage[T1]{fontenc}
\usepackage[utf8]{inputenc}
\usepackage{enumitem}
\usepackage{bm}
\usepackage[usenames, dvipsnames]{color}
\usepackage{listings}
\usepackage{siunitx}
\usepackage[draft=false,linktocpage=true]{hyperref}
\usepackage[capitalise]{cleveref}
\usepackage{comment}
\usepackage{calligra}
\usepackage{setspace}
\usepackage[final]{microtype}
\usepackage{tikz,tikz-cd}
\usetikzlibrary{trees,calc, arrows, decorations.markings}
\usepackage{graphicx}
\usepackage{wrapfig}
\usepackage{float}
\usepackage{caption}
\usepackage{subcaption}
\numberwithin{equation}{section}
\newtheorem{theorem}{Theorem}[section]
\newtheorem{theodef}[theorem]{Theorem/Definition}
\newtheorem{proposition}[theorem]{Proposition}
\newtheorem{lemma}[theorem]{Lemma}
\newtheorem{definition}[theorem]{Definition}
\newtheorem{corollary}[theorem]{Corollary}

\newtheorem{remark}[theorem]{Remark}

\DeclareMathOperator{\Hom}{Hom}

\DeclareMathOperator{\End}{End}
\DeclareMathOperator{\Sym}{Sym}

\DeclareMathOperator{\iHom}{\mathscr{H}\text{\kern -3pt {\calligra\large om}}\,}
\DeclareMathOperator{\iEnd}{\mathscr{E}\text{\kern -3pt {\calligra\large nd}}\,}
\DeclareMathOperator{\iExt}{\mathscr{E}\text{\kern -3pt {\calligra\large xt}}\,}
\DeclareMathOperator{\iTor}{\mathscr{T}\text{\kern -3pt {\calligra\large or}}\,}

\DeclareMathOperator{\colim}{colim}

\DeclareMathOperator{\HH}{HH_{\bm{\cdot}}}

\tikzset{%
    symbol/.style={%
        draw=none,
        every to/.append style={%
            edge node={node [sloped, allow upside down, auto=false]{$#1$}}}
    }
}

\newcommand{\abs}[1]{\left\lvert#1\right\rvert}
\newcommand{\Mod}[1]{#1\textrm{-}\textsf{Mod}}

\newcommand{\Spec}[1]{\mathrm{Spec} \, #1}
\newcommand{\Spet}[1]{\mathrm{Sp}\textrm{ét} \, #1}

\newcommand\opposite[1]{{#1}^{\mathrm{op}}}
\newcommand{\id}{\mathrm{id}}
\newcommand{\QC}[1]{\mathrm{QCoh(#1)}}

\title{Grothendieck Duality via Diagonally Supported Sheaves}
\author{Andy Jiang}
\address{Department of Mathematics,University of Michigan, 530 Church Street,
  Ann Arbor, MI 48109}
\email{ndjiang@umich.edu}

\begin{document}
\begin{abstract}
  Following a formula found in the \cite{AILN} and ideas of \cite{N2} and \cite{Khus}, we indicate that Grothendieck duality for finite tor-amplitude
  maps can be developed from scratch via the
  formula $f^! := \delta^*\pi_1^{\times}f^*$. Our strategy centers on the subcategory $\Gamma_{\Delta}(\QC{X \times X})$ of quasicoherent
  sheaves on $X \times X$ supported on the diagonal. By exclusively using this subcategory instead of the full category $\QC{X \times X}$ 
  we give systematic categorical proofs of results in Grothendieck duality and reprove many formulas found in \cite{N1}. We also relate
  some results in Grothendieck duality with properties of the sheaf of (derived) Grothendieck differential operators. 
\end{abstract}
\maketitle

\section*{Introduction}
\textbf{Motivation}:
Suppose $X$ is a finite-type, separated, flat
scheme over a Noetherian base scheme $S$, Proposition 3.3 of \cite{N1} (building on Theorem 4.6 of \cite{AILN} and Lemma 3.2.1 of \cite{ILN}),
shows the isomorphism
\begin{equation} \label{shriekformula}
  f^!_c \cong \delta^*\pi_1^{\times}f^*
\end{equation}
where the maps are defined as in the diagram
\begin{equation}\label{eq3}
  \begin{tikzcd}
    X \arrow{dr}{\delta} \\
    & X \times_S X  \arrow{r}{\pi_2} \arrow{dd}{\pi_1} & X \arrow{dd}{f} \\ \\
    & X \arrow{r}{f} & S\\
  \end{tikzcd}
\end{equation}
Here, $f^!_c$ denotes the exceptional pullback functor in Grothendieck duality, defined in a classical way, and
$\pi_1^{\times}$ denotes the right adjoint to the pushforward functor $\pi_{1,*}$. This formula has the advantage over classical definitions that
it does not depend on a choice of a compactification of $X$. Inspired by this,
one may ask if it is possible to develop Grothendieck Duality from scratch using this formula, and thus bypassing the issue of compactifications.
This was the approach taken by the thesis of Hafiz Khusyairi \cite{Khus}, which proved many properties of (\ref{shriekformula}) in the situation of flat morphisms, as above.
In \cite{N2}, Neeman extends this work and gives some indication that (\ref{shriekformula}) can be used as a foundation for Grothendieck duality--proving Serre duality without resorting
to any existing theory of Grothendieck duality.

One source of complication for developing Grothendieck duality using (\ref{shriekformula}) appears in Section 4.2 of \cite{N2}.
We need to show that the right hand side of (\ref{shriekformula}) is local on $X$. Namely,
if we write $f^!_r$ for the right hand side of (\ref{shriekformula}) and $u : U \to X$ is an open immersion, we need to show that
\[u^*f^!_r \cong (uf)^!_r\]
The majority of Section 4.2 of \cite{N2} is devoted to a proof of this fact. In this paper, we try to provide
a more conceptual framework to understand statements like this one and their proofs, by relying everywhere on the category $\Gamma_{\Delta}(\QC{X \times_S X})$.
We note that the category $\Gamma_{\Delta}(\QC{X \times_S X})$ appears in Neeman's writing as well. However, though he makes use of it
more sparingly, we aim to use this category whenever possible. In an upcoming work \cite{upcoming}, we will discuss this category as a categorified version
of the ring of differential operators, following the works of \cite{Ber1} and \cite{Ber2}.

The key point is that the formula $\delta^*\pi_1^{\times}f^*$ is not inherently local on $X$, due to the appearance of the (non colimit-preserving) functor $\pi_1^{\times}$.
Additionally, the sheaf of categories $U \mapsto \QC{U \times_S U}$ is not a quasicoherent sheaf of categories on $X$ (in a precise sense which we define in Proposition \ref{quasicoherent}).
However, as we will see in Proposition \ref{quasicoherent},
the sheaf $U \mapsto \Gamma_{\Delta}(\QC{U \times_S U}$ is, where the latter category is the full subcategory of $\QC{U \times_S U}$ supported on the diagonal.
Additionally, as $\delta^*$ only sees the part of the quasicoherent sheaf on $X \times_S X$
that is supported on the diagonal, we can actually rewrite $\delta^*\pi_1^{\times}f^*$ in a way which bypasses the category $\QC{X \times X}$.
Namely,
\[\delta^*\pi_1^{\times}f^* \cong \tilde{\delta}^*\tilde{\pi}_1^{\times}f^*\]
where
\[\tilde{\delta}^* : \Gamma_{\Delta}(\QC{X \times_S X}) \to \QC{X}\]
and
\[\tilde{\pi}^{\times}_1 : \QC{X} \to \Gamma_{\Delta}(\QC{X \times_S X})\]
are analogues of $\delta^*$ and $\pi_1^{\times}$ involving only $\Gamma_{\Delta}(\QC{X \times_S X})$ (see Section 2 for the precise definitions).
Therefore, we achieve a rewriting
of $f^!_r$ which is \textit{manifestly} local. All the technical inputs are cleanly packaged into two statements:
\begin{enumerate}
\item{$U \mapsto \Gamma_{\Delta}(\QC{X \times_S X})$ is quasi-coherent sheaf of categories} (Proposition \ref{quasicoherent})
\item{$\tilde{\pi}_1^{\times}$ is a quasicoherent map} (Proposition \ref{qcoh2})
\end{enumerate}

In fact, more can be said about the functor $\tilde{\pi}^{\times}_1$. It is related to the theory of $D$-modules, in a way which we will explain in the
``Relation to $D$-modules''
subsection below.

\textbf{Results}: 
Let $X$ be a spectral Noetherian scheme which is separated, locally almost of finite presentation,
and finite tor-amplitude over a base homologically bounded spectral Noetherian scheme $S$ throughout. We remind the reader
that the fibre products of schemes will be implicitly derived from now on.

We show that the category $\Gamma_{\Delta}(\QC{X \times_S X})$ is a quasicoherent sheaf of categories on $X$ in the following sense
(see Proposition \ref{quasicoherent} in the main text).
\begin{proposition}
  For an étale map $u : U \to X$, we have
  \[\QC{U} \otimes_{\QC{X}} \Gamma_{\Delta}(\QC{X \times_S X}) \cong \Gamma_{\Delta}(\QC{U \times_S U})\]
  where $\QC{X}$ acts on $\QC{U}$ via $j^*$ and $\QC{X}$ acts on $\Gamma_{\Delta}(\QC{X \times_S X})$
  via $\Gamma_{\Delta}\pi_1^*$.
\end{proposition}
This in particular shows that $\Gamma_{\Delta}(\QC{X \times_S X})$ satisfies étale descent.
Next, we show that the functor $\tilde{\pi}_1^{\times}$ is a quasicoherent map of categories for formal reasons, see Proposition \ref{qcoh2}
in the main text.
\begin{proposition}
  For an étale map $u : U \to X$, 
  \[\tilde{\pi}_{1,U}^{\times} :\QC{U} \to \Gamma_{\Delta}(\QC{U \times_S U})\]
  is $\QC{U}$ linear and
  agrees with $\tilde{\pi}_{1,X}^{\times}$ for $X$ base changed to $U$, i.e. tensored with $\QC{U}$ over $\QC{X}$.
\end{proposition}

The two propositions above provide the backbone for our results developing Grothendieck duality
using \ref{shriekformula}. We start with the definition (Definition \ref{defshriek} in the main text),

\begin{definition}
  Suppose $f : X \to S$ is a map of spectral Noetherian schemes which is finite tor-amplitude, locally almost of finite-presentation, and separated,
  and that both $X$ and $S$ are homologically bounded. Then, we define
  \[f^! := \delta^*\pi_1^{\times}f^* : \QC{S} \to \QC{X}\]
  where the maps are as shown in the diagram (\ref{eq3}).
\end{definition}

We prove the exceptional pullback (also referred to as upper shriek) functor defined above satisfies the following properties. The following is contained in
Equation (\ref{eq2}), Corollary \ref{csk}, Proposition \ref{csk2}, and Theorem \ref{csk3} in the text.
Parts of this theorem are contained in \cite{N2}, but we take a slightly different approach.
\begin{theorem}
  Suppose $f : X \to S$ is a map of homologically bounded
  spectral Noetherian schemes which is finite tor-amplitude, locally almost of finite-presentation, and separated.
  Then,
  \begin{enumerate}
    \item $f^!$ is colimit-preserving (in fact $\QC{S}$-linear)
    \item If $f$ is proper, then $f^! \cong f^{\times}$.
    \item If $f$ is étale, then $f^! \cong f^*$.
    \item If $g: X' \to X$ is also finite tor-amplitude, locally almost of finite-presentation, and separated, then,
      \[g^!f^! \cong (fg)^!\]
    \end{enumerate}
\end{theorem}

We also demonstrate the versatility of our techniques by showing some ``Hochshild-style'' formulas from \cite{N1},
see Theorem \ref{hochschildstyle} in the text.
\begin{theorem} 
  For $\mathcal{F} \in \QC{S}$ and $\mathcal{G} \in \QC{X}$, we have
  \[\delta^{\times}\pi_1^*\iHom(p_X^*\mathcal{F},\mathcal{G}) \cong \iHom(p_X^!\mathcal{F},\mathcal{G})\]
  where $\iHom$ denotes internal Hom of quasicoherent sheaves.
\end{theorem}

\textbf{Relation to $D$-modules}: 
For simplicity, let us now assume $S=\Spec{k}$ for a field $k$ and $X=\Spec{A}$,
a smooth affine variety over a field $k$. The ring of Grothendieck differential operators on $X$ relative to $k$, $D_{X/k}$,
is defined to be the increasing union
\[D_{X/k} := \bigcup_{n \ge 0} {D^{(n)}} \subseteq \Hom_k(A,A)\]
where $D^{(n)} \subseteq \Hom_k(A,A)$ defined inductively by
\[D^{(-1)}=0\]
and
\[D^{(n)}=\{f \in \Hom_k(A,A) | \forall a \in A,  [f,a] \in D^{(n-1)}\}\]
($a \in A$ is thought of as an element $\Hom_k(A,A)$ via multiplication by $a$)

We can in fact interpret the above definition as follows. $\Hom_k(A,A)$ is naturally a $A$-bimodule, with
the left action given by post-composition of a $k$-linear function from $A$ to $A$ with multiplication by an element
of $A$, and the right action given by pre-composition. Then,
\[D^{(n)} \cong V(\mathcal{I}^n)\]
is the subset of $\Hom_k(A,A)$ which is annihilated by the ideal $I^n$, where
\[I:=(\{a \otimes 1 - 1 \otimes a | a \in A\})\] is the ideal of $A \otimes A$
of regular functions on $X \times X$ which vanishes on the diagonal.
Therefore, we see that $D_{A/k}$ is just the subset of $\Hom_k(A,A)$
on which $I$ acts nilpotently, or equivalently,
\[D_{A/k} \cong H^0(\Gamma_{\Delta}(\Hom_k(A,A))\]
the $0$-th local cohomology of the module $\Hom_k(A,A)$ along the diagonal of $X$.
In fact, Proposition 2.2.1 of \cite{SVdB} shows that the higher cohomologies vanish for smooth $A$.

Given this, it is natural to define in general the Grothendieck
ring of differential operators on an affine (not necessarily smooth) scheme $X=\Spec{A}$ over $\Spec{k}$ as
\[D_{A/k}:= \Gamma_{\Delta}(\Hom_k(A,A))\]
This relates to the functor $\tilde{\pi}_1^{\times}$ by
\begin{equation}
  \begin{split}
    D_{A/k} &:= \Gamma_{\Delta}(\Hom_k(A,A)) \\
    &\cong \Gamma_{\Delta}(\Hom_A(A \otimes_k A,A)) \\
    &\cong \Gamma_{\Delta}(\pi_1^{\times}(A)) \\
    &\cong \tilde{\pi}_1^{\times}(A)
  \end{split}
\end{equation}
which is basically the definition we give in the text in Definition \ref{bigdef}.
Additionally, the above relates to a classical formula for the ring of differential operators by
\begin{equation}
  \begin{split}
    D_{A/k} 
    &\cong \Gamma_{\Delta}(\Hom_A(A \otimes_k A,A)) \\
    &\cong \colim_n(\Hom_A(A \otimes_k A/I^n,A) \\
  \end{split}
\end{equation}
where here our Hom is derived.

For the reasons above, the question of locality of the exceptional pullback functor is related to the locality
of the sheaf of Grothendieck differential operators, which explains why we begin with describing properties of the sheaf
of differential operators in Section \ref{G}.

\textbf{Outline}: Section 1 introduces the sheaf of Grothendieck
differential operators, and various properties of the category $\Gamma_{\Delta}(\QC{X \times X})$.
Section 2 defines the exceptional pullback/upper shriek functor and proves its properties. Section 3
discusses two fundamental dualities which show up in our picture of Grothendieck duality (and will be
important in upcoming work in $D$-modules \cite{upcoming}). Section 4 defines a quasicoherent lower shriek functor (different
to étale lower shriek!) and reproves many ``Hochschild-style'' formulas from \cite{N1}.
Section 5 discusses our constructions in the smooth case, where we can compute the objects explicitly.

\textbf{Conventions}: All categories, unless stated otherwise will be $(\infty, 1)$-categories. A $2$-category
will refer to a $(\infty, 2)$-category. All functors, such as $\Hom$, $\otimes$, $\colim$, and $\lim$ will be
fully derived/done at the $\infty$-categorical level unless stated otherwise. A stable category will refer to
a stable $\infty$-category. All modules/quasicoherent sheaves will also be assumed to be fully derived.
We will aim to follow the terminology of Lurie in \cite{HTT}, \cite{HA}, and \cite{SAG}.

\textbf{Acknowledgements}: I heartily thank my advisor, Bhargav Bhatt, for countless discussions, suggestions, 
and insights whose effects permeate this paper. Thanks also to Shubhodip Mondal and Sridhar Venkatesh for helpful discussions.
I would also like to thank German Stefanich for patiently explaining
to me many ideas of the papers of Dario Beraldo, \cite{Ber1} and \cite{Ber2}.

\section{Grothendieck Differential Operators} \label{G}
In this section, we define the sheaf of Grothendieck differential operators $D_X$ and show that it satisfies
étale descent. The main ingredient is the ``quasicoherent'' sheaf of categories $\Gamma_{\Delta}(\QC{X \times X})$,
which we discuss at length in this section. However, discussion of the ring structure on $D_X$ is deferred to an upcoming paper \cite{upcoming},
where we will also discuss developing the theory of $D$-modules from scratch using the (derived) ring $D_X$. These
results about the sheaf $D_X$ will allow us to obtain locality results of the upper shriek functor in Section \ref{upshriek}.

We adopt the terminology of \cite{SAG}.
Fix a homologically bounded spectral Noetherian base scheme $S$. Consider
a spectral Noetherian schemes $X$ with a structure map $p_X : X \to S$
which is separated, locally almost of finite presentation, and finite tor-amplitude. 
All our spectral schemes will therefore be perfect stacks, in the terminology of \cite{SAG} (note that
this is less restricted than the notion introduced in \cite{BZFN} with the same name).

Let $X \times_S X$ be the pullback
of $p$ along itself. We define $\pi_1$ and $\pi_2$ to be the two projection
maps of this pullback. Here is a diagram,
\begin{equation}\label{eq1}
  \begin{tikzcd}
    X \times_S X  \arrow{r}{\pi_2} \arrow{dd}{\pi_1} & X \arrow{dd}{p} \\ \\
    X \arrow{r}{p} & S\\
  \end{tikzcd}
\end{equation}

Let $\Delta$ denote the diagonal inside $X \times_S X$ (which is abstractly isomorphic to $X$).
Then we take the following as the definition
of the Grothendieck sheaf of differential operators.

\begin{definition} \label{bigdef}
  The Grothendieck sheaf of differential operator of $X$ over $S$ is defined to be
  \[D_{X/S} := \Gamma_{\Delta} \pi_1^{\times}\mathcal{O}_X \in \Gamma_{\Delta}(\QC{X \times_S X})\]
  where $\pi_1^{\times}$ is the right adjoint of the pushforward functor $\pi_{1,*}$ and
  $\Gamma_{\Delta}$ is the right adjoint of the inclusion functor
  \[i_{\Delta}:\Gamma_{\Delta}(\QC{X \times_S X}) \to \QC{X \times_S X}\]
  of the subcategory of diagonally supported quasicoherent sheaves.
  Often we will suppress $S$ from the notation and write simply $D_X$.
\end{definition}

\begin{remark} \label{affinediff}
  In the affine case, where $X \cong \Spec{A}$ and $S \cong \Spec{k}$, the above simply says that $D_{X/S} \cong \Gamma_{\Delta}\Hom_k(A,A)$.
  where $D_{X/S}$ is thought of as a $A$-bimodule supported on the diagonal.
  This is because the functor $\pi_{1,*} : \Mod{A \otimes_k A} \to $ is given by the formula
  \[\pi_{1,*}(M) \cong (A \otimes_k A) \otimes_{A \otimes_k A} (M)\]
  where $A$ acts on the left $A$ in the tensor and $A \otimes A$ acts on $A \otimes A$ by
  multiplication inside $A \otimes A$. Therefore its right adjoint is given by the formula
  \[\pi_1^{\times}(M) \cong \Hom_A(A \otimes_k A, M)\]
  By adjunction this is the same as $\Hom_k(A,M)$ where the $A \otimes_k A$ module structure
  has the left $A$ acting on $M$ (the codomain) and the right $A$ acting on $A$ (the domain). This
  means the left $A$ acts by postcomposition of the $k$-linear function with multiplication
  by an element of $A$ and the right $A$ acts by precomposition. Visually, we have
  \[((a_1 \otimes a_2)f)(x) = a_1 f ( a_2 x)\]
  for $f \in \Hom_k(A,M)$.
\end{remark}

We write
\[\tilde{\pi}_{1,*} : \Gamma_{\Delta}(\QC{X \times_S X}) \to \QC{X}\]
for the composition $\pi_{1,*}i_{\Delta}$ and
\[\tilde{\pi}_1^{\times} : \QC{X} \to \Gamma_{\Delta}(\QC{X \times_S X})\]
for the right adjoint of $\tilde{\pi}_{1,*}$. Then $D_{X/S} \cong \tilde{\pi}_1^{\times}\mathcal{O}_X$.
Importantly, $\tilde{\pi}_1^{\times}$ is a colimit-preserving functor. This follows from the following theorem
because all the categories in sight are compactly generated (see Lemma \ref{compactgen}).

\begin{theorem} \label{AGinput1}
  $\tilde{\pi}_{1,*}$ preserves compact objects.
\end{theorem}
\begin{proof}
  Because $i_{\Delta}$ preserves compact objects, any compact object
  \[x \in \Gamma_{\Delta}(\QC{X \times_S X})\]
  can be thought of as a compact object $x \in \QC{X \times_S X}$ supported at the diagonal.
  By SAG 5.6.5.2, $\pi_{1,*}(x)$ is almost perfect. Because $p_X: X \to S$ is finite tor-amplitude, $\pi_{1,*}(x)$
  is also finite tor-amplitude. Therefore, $\pi_{1,*}(x)$ is perfect by HA 7.2.4.23, hence compact.
\end{proof}
\begin{corollary} \label{AGcoro1}
  $\tilde{\pi}_1^{\times}$ is a colimit-preserving $\QC{X}$-linear functor, where $\QC{X}$ acts on $\Gamma_{\Delta}(\QC{X \times X})$
  via $\Gamma_{\Delta}\pi_1^*$.
\end{corollary}
\begin{proof}
  Follows from the theorem above and Theorem \ref{Proj1}.
\end{proof}

To prove étale descent results about $D_{X/S}$ let us first show some étale descent results about
the category $\Gamma_{\Delta}(\QC{X \times_S X})$. Namely we want to show that $\Gamma_{\Delta}(\QC{X \times_S X})$
is a ``quasicoherent'' sheaf of categories over $\QC{X}$ with action via $\Gamma_{\Delta}\pi_1^*$.

\begin{proposition} \label{quasicoherent}
  For an étale map $u : U \to X$, we have
  \[\QC{U} \otimes_{\QC{X}} \Gamma_{\Delta}(\QC{X \times_S X}) \cong \Gamma_{\Delta}(\QC{U \times_S U})\]
  where $\QC{X}$ acts on $\QC{U}$ via $j^*$ and $\QC{X}$ acts on $\Gamma_{\Delta}(\QC{X \times_S X})$
  via $\Gamma_{\Delta}\pi_1^*$.
\end{proposition}
\begin{proof}
  The left hand side is canonically
  \[\QC{U \times_S X} \otimes_{\QC{X \times_S X}} \Gamma_{\Delta}(\QC{X \times_S X}) \cong \Gamma_{\Delta}(\QC{U \times_S X})\]
  because tensor products preserve split-exact sequences of presentable stable categories (see Proposition \ref{tensorref}). Now the result follows from the
  Proposition \ref{AGinput2} applied to the diagonal closed immersion $U \to U \times X$ which factors through the map $U \times U \to U \times U$.
\end{proof}

\begin{corollary} \label{talesheaf}
  \[U \mapsto \Gamma_{\Delta}(\QC{U \times_S U})\]
  is an étale sheaf on $X$.
\end{corollary}
\begin{proof}
  This follows from the above proposition as all quasicoherent sheaves of categories satisfy étale descent (see Remark 10.1.2.10 of
  \cite{SAG}), though in this case it is easy to check directly that $\Gamma_{\Delta}(U \times_S X)$ is an étale sheaf directly as well.
\end{proof}

\begin{remark}
  Proposition \ref{quasicoherent} and Corollary \ref{talesheaf} admit obvious generalizations to products of more than two terms.
\end{remark}

Next, we show that $\tilde{\pi}_1^{\times}$ is a ``quasicoherent'' map of (quasicoherent) sheaves of categories. 
\begin{proposition} \label{qcoh2}
  For an étale map $u : U \to X$, 
  \[\tilde{\pi}_{1,U}^{\times} :\QC{U} \to \Gamma_{\Delta}(\QC{U \times_S U})\]
  is $\QC{U}$ linear and
  agrees with $\tilde{\pi}_{1,X}^{\times}$ for $X$ base changed to $U$, i.e. tensored with $\QC{U}$ over $\QC{X}$.
\end{proposition}
\begin{proof}
  The map above is $\QC{U}$ linear by Corollary \ref{AGcoro1}.
  The second claim above follows because tensoring with $\QC{U}$ over $\QC{X}$ preserves adjoints of colimit-preserving functors and
  $\tilde{\pi}_{1,*}$ for $X$ tensored to $U$ agrees with $\tilde{\pi}_{1,*}$ for $U$.
\end{proof}

For $\mathcal{F} \in \Gamma_{\Delta}(\QC{X \times_S X})$, and an étale map $u : U \to X$,
we denote by $\mathcal{F}|_U$ the (quasicoherent) pullback of $\mathcal{F}$ in
\[\Gamma_{\Delta}(\QC{U \times_S U})\]

\begin{proposition} \label{prop14}
  For $\mathcal{F} \in \QC{X}$,
  \[\tilde{\pi}_{1,X}^{\times}(\mathcal{F})|_{U} \cong \tilde{\pi}_{1,U}^{\times}(\mathcal{F}|_U)\]
\end{proposition}
\begin{proof}
  This is a direct consequence of Theorem \ref{BC2} applied to pullback along $u:U \to X$ and upper cross functor $\tilde{\pi}_1^{\times}$
  and the above proposition.
\end{proof}
\begin{corollary}
  \[D_{X/S}|_U \cong D_{U/S}\]
\end{corollary}

It is easily checked that the entire story behaves well with respect to base-change in $S$. For example, suppose we have a map $q: S' \to S$
of spectral Noetherian schemes, then we can consider $X'= X \times_S S'$ living over $S'$. The base-change
of $D_{X/S}$ to $\Gamma_{\Delta}(\QC{X' \times_{S'} X'})$ is then $D_{X'/S'}$.
  
As in the classical definition, which we will later show agrees with our definition in the smooth case,
$D_X$ carries a ring structure. However, we will leave this to an upcoming work \cite{upcoming}.
Nevertheless we note here there is a natural convolution monoidal structure on
$\Gamma_{\Delta}(\QC{X \times_S X})$.

\begin{theodef} \label{convprod}
  There is a natural convolution monoidal structure on $\Gamma_{\Delta}(\QC{X \times_S X})$.
  $\QC{X}$ has the structure of a left $\Gamma_{\Delta}(\QC{X \times X})$ module with respect to this
  monoidal structure.
\end{theodef}
\begin{proof}
  We construct it by inducing it from $\QC{X \times_S X}$. We have the isomorphism
  \[\QC{X \times_S X} \cong \End_S(\QC{X},\QC{X})\]
  The right hand side is endomorphisms of $\QC{X}$ inside $\Mod{\QC{S}}^L$, and thus
  has a natural monoidal structure.
  This gives the desired monoidal structure on $\QC{X \times_S X}$.
  It is easy to see that $\Gamma_{\Delta}(\QC{X \times_S X})$
  is closed under this product, and so inherits a convolution monoidal structure. The second
  part of the theorem is clear from our construction.
  
  To be explicit, given two quasicoherent sheaves $F$ and $G$ on $X \times X$, their convolution is simply
  \[F \star G := \pi_{1,3,*}(\pi_{1,2}^*F \otimes \pi_{2,3}^*G)\]
  where
  \[\pi_{i,j} : X \times X \times X \to X \times X\]
  are the obvious projection maps.
\end{proof}
\begin{remark}
  If $X = \Spec{A}$, this tensor product for $A$-bimodules is simply given by tensoring the two $A$-bimodules together over $A$.
\end{remark}

For $U$ étale over $X$, one can easily check that the pullback map
\[\Gamma_{\Delta}(\QC{X \times_S X}) \to \Gamma_{\Delta}(\QC{U \times_S U})\]
is monoidal with respect to the convolution product.
In an upcoming paper \cite{upcoming}, we will show that $D_{X/S}$ is an algebra with respect
to the convolution monoidal product defined above. 

\section{Dualizing Complexes and the Upper Shriek Functor} \label{upshriek}
This section is dedicated to defining the upper shriek functor and proving some properties of it.
Almost all of the results in this section are, in some form,
contained in \cite{ILN} and \cite{N1}. The key differences are the order of presentation--we
define the upper shriek functor without compactifications at all and develop its properties from scratch--and
the fact that we make heavy use of the category $\Gamma_{\Delta}(\QC{X \times_S X})$, which is morally ``proper'' over $X$.
We are motivated to study this subcategory for its own sake in view of its connections with differential operators.

We begin by defining the upper shriek functor 
for a map $p_X : X \to S$ with the assumptions from the previous section.
\begin{definition} \label{defshriek}
  The upper shriek functor $p_X^!: \QC{S} \to \QC{X}$ is defined by
  \[p_X^!:= \delta^*\pi_1^{\times}p_X^*\]
  where $\delta : X \to X \times_S X$ is the diagonal map.
\end{definition}
This formula for the upper shriek functor appears in many places in the literature, e.g. \cite{N1} Proposition 3.3,
however here we will take it as a definition.
The main property of upper shriek is that it behaves well under composition, that is
\[(fg)^! \cong g^!f^!\]
and that it interpolates between upper-cross pullback in the proper case and upper-star pullback in the étale case. This is
what we aim to show in this section.

The pullback functor along the diagonal
\[\delta^* : \QC{X \times_S X} \to \QC{X}\]
factors through the local cohomology functor $\Gamma_{\Delta}$, namely
\[\delta^* \cong \tilde{\delta}^*\Gamma_{\Delta}\]
Therefore,
\begin{equation} \label{eq2}
  p_X^! \cong \tilde{\delta}^*\tilde{\pi}_1^{\times}p_X^*
\end{equation}
From the above we see the upper shriek functor is colimit preserving and $\QC{S}$-linear.
\begin{proposition} \label{csk2}
  Suppose $u : U \to X$ is a separated étale map, then
  \[u^! = u^*\]
\end{proposition}
\begin{proof}
  This follows from the Proposition \ref{AGinput2} applied to the closed immersion $U \to U \times_X X$ with a lift to
  $U \times_X U$. Namely, we know that
  \[\tilde{\pi}_{1,*} : \Gamma_{\Delta}(\QC{U \times_X U}) \to \QC{U}\] is an isomorphism and its inverse and adjoint (on both sides) is
  $\Gamma_{\Delta}\pi_1^*$ (where $\pi_1 : U \times_X U \to U$ is the projection map to the first component). Therefore
  \begin{equation*}
    \begin{split}
      u^! 
      &\cong \tilde{\delta}^*\tilde{\pi}_1^{\times}u^* \\
      &\cong \tilde{\delta}^*\Gamma_{\Delta}\pi_1^*u^* \\
      &\cong u^* \\
    \end{split}
  \end{equation*}
\end{proof}

\begin{proposition} \label{openshriek}
  Suppose $u : U \to X$ is a separated étale map, and $p_U : U \to S$ is the structure map, then
  \[p_U^! \cong u^*p_X^! \cong u^!p_X^! \]
\end{proposition}
\begin{proof}
  The second isomorphism follows from the previous proposition, the first follows because
  \begin{equation*}
    \begin{split}
      p_U^! \mathcal{F}
      &\cong \tilde{\delta}_U^* \tilde{\pi}_{1,U}^{\times} p_U^* \mathcal{F}\\
      &\cong \tilde{\delta}_U^* \tilde{\pi}_{1,U}^{\times} u^* p_X^* \mathcal{F}\\
      &\cong \tilde{\delta}_U^* (\tilde{\pi}_{1,X}^{\times} p_X^* \mathcal{F})|_U\\
      &\cong u^*\tilde{\delta}_X^* \tilde{\pi}_{1,X}^{\times} p_X^* \mathcal{F}\\
    \end{split}
  \end{equation*}
  where the third isomorphism uses Proposition \ref{prop14}.
\end{proof}

\begin{definition} \label{dualizedef}
  We define the dualizing complex of a scheme $X$ by
  \begin{equation}
    \omega_{X/S} := p_X^!(\mathcal{O}_S) \cong \tilde{\delta}^*D_{X/S} 
  \end{equation}
\end{definition}
We note that in the affine case we simply have
\[\omega_{A/k} \cong D_{A/k} \otimes_{A \otimes_k A} A \cong \Hom_k(A,A) \otimes_{A \otimes_k A} A\]
Because of the Proposition \ref{openshriek}, for a separated étale $u : U \to X$, we have
\[u^*\omega_{X/S} \cong \omega_{U/S}\]
Also, $\omega$ behaves well under base-change with respect to $S$. Namely, if $q: S' \to S$ is a map
of spectral Noetherian schemes, there is an isomorphism
\[\omega_{X \times_S S'/S'} \cong (\id \times q)^*(\omega_{X/S})\]
To be general,
\begin{theorem}\label{shriekbase}
  Suppose we have the following pullback diagram (of homologically bounded separated spectral Noetherian schemes)
  \[
    \begin{tikzcd}
      Y_{S'} \cong Y \times_S S'  \arrow{r}{\pi_2} \arrow{dd}{\pi_1} & Y \arrow{dd}{p_Y}\\ \\
      S' \arrow{r}{p_{S'}} & S \\
    \end{tikzcd}
  \]
  where $p_Y$ is finite tor-amplitude and all maps are locally almost of finite presentation.
  Then
  \[\pi_1^!p_{S'}^* \cong \pi_2^*p_Y^!\]
\end{theorem}
\begin{proof}
  The entire construction base-changes well with respect to $S$, so this is clear.
\end{proof}

Because $p_X^!$ is colimit preserving and $\QC{S}$-linear, we have
\begin{equation} \label{shriekeq}
  p_X^!(\mathcal{F}) \cong \omega_{X/S} \otimes p_X^*\mathcal{F}
\end{equation}

\begin{remark}
  Equation (\ref{shriekeq}), combined with the analogous statement for the classically defined
  upper-shriek functor (see \cite{Basechangeneeman} Remark 1.22) implies via Corollary 4.7 of \cite{AILN} that our upper-shriek functor
  agrees with the classical one for finite tor-amplitude, finite-type, separated morphisms
  of non-derived Noetherian schemes.
\end{remark}

\begin{proposition} \label{crossshriek}
  There is a natural transformation
  \[p_X^{\times} \to p_X^!\]
  and hence also a natural map
  \[\Hom(p_{X,*}\mathcal{F},\mathcal{G}) \to \Hom(\mathcal{F},p_X^!\mathcal{G})\]
  for $\mathcal{F} \in \QC{X}$ and $\mathcal{G} \in \QC{S}$.
\end{proposition}
\begin{proof}
  Consider the square
  \[
    \begin{tikzcd}
      X \times_S X  \arrow{r}{\pi_2} \arrow{dd}{\pi_1} & X \arrow{dd}{p} \\ \\
      X \arrow{r}{p} & S\\
    \end{tikzcd}
  \]
  From the push-pull isomorphism, there is a map 
  \[\pi_{1,*}\pi_2^*p_X^{\times} \cong p_X^*p_{X,*}p_X^{\times} \to p_X^*\]
  hence by adjunction, we have a map
  \[\pi_2^{*}p_X^{\times} \to \pi_1^{\times}p_X^*\]
  Pulling back along $\delta$ gives a map
  \[p_X^{\times} \to \delta^*\pi_1^{\times}p_X^* \cong p_X^!\]
\end{proof}

Suppose $g: Y \to X$ over $S$ is finite tor-amplitude such that the composition $p_x \circ g$ is proper,
then we have (this result is Lemma 3.1 in \cite{N1}).
\begin{theorem} \label{crossshriek2}
  The natural transformation $p_X^{\times} \to p_X^!$ is an isomorphism after post-composition with $g^{\times}$.
\end{theorem}
\begin{proof}
  Consider the diagram
  \[
    \begin{tikzcd}
      Y \times_S X  \arrow{r}{g \times \id} \arrow{dd}{\pi'_1} & X \times_S X \arrow{r}{\pi_2} \arrow{dd}{\pi_1} & X \arrow{dd}{p_X} \\ \\
      Y \arrow{r}{g} & X \arrow{r}{p_X} & S \\
    \end{tikzcd}
  \]
  The outer rectangle exhibits pull-cross base-change (Theorem \ref{BC2}), namely,
  \[\pi_1^{'*} g^{\times} p_X^{\times} \cong (g \times \id)^{\times} \pi_2^{\times} p_X^*\]
  The map exhibiting the isomorphism
  is formed using the pull-cross base-change maps for the two smaller squares.
  Now we post-compose the above isomorphism with the
  pullback along the graph of $g$, $\delta_g : Y \to Y \times_S X$, to get
  \[g^{\times} p_X^{\times} \cong \delta_g^*(g \times \id)^{\times} \pi_2^{\times} p_X^*\]
  Now looking at the pull-cross base-change for the diagram (since $g$ is also proper)
  \[
    \begin{tikzcd}
      Y \arrow{r}{\delta_g} \arrow{dd}{g} & Y \times_S X \arrow{dd}{g \times \id} \\ \\
      X \arrow{r}{\delta} & X \times_S X\\
    \end{tikzcd}
  \]
  namely,
  \[g^{\times}\delta^* \cong \delta_g^*(g \times \id)^{\times}\]
  We have,
  \[g^{\times} p_X^{\times} \cong \delta_g^*(g \times \id)^{\times} \pi_2^{\times} p_X^*
    \cong g^{\times}\delta^*\pi_2^{\times} p_X^* \cong g^{\times}p_X^!\]
  One checks that the map agrees with the map in the previous proposition
  post-composed with $g^{\times}$ by staring at the following combined diagram using
  the fact that the base-change for the left tall rectangle is trivial.
  \[
    \begin{tikzcd}
      Y  \arrow{r}{g} \arrow{dd}{\delta_g} & X \arrow{dd}{\delta}  \\ \\
      X \times_S Y  \arrow{r}{g \times \id} \arrow{dd}{\pi'_1} & X \times_S X \arrow{r}{\pi_2} \arrow{dd}{\pi_1} & X \arrow{dd}{p_X} \\ \\
      Y \arrow{r}{g} & X \arrow{r}{p_X} & S \\
    \end{tikzcd}
  \]
\end{proof}

\begin{corollary} \label{csk}
  Suppose $p_X : X \to S$ is proper, then $p_X^{\times} \cong p_X^!$.
\end{corollary}

\begin{theorem} \label{crossshriek4}
  Suppose $\Lambda$ is a closed subset of $\abs{X}$ which is proper over $S$, then
  \[\Gamma_{\Lambda}p_X^{\times} \cong \Gamma_{\Lambda}p_X^!\]
\end{theorem}
\begin{proof}
  Repeat the argument used to prove Theorem \ref{crossshriek2}, rephrased in terms of categories of quasicoherent sheaves and then substitute $\Gamma_{\Lambda}(\QC{X})$
  wherever $\QC{Y}$ appears, using the fact that
  $p_{X,*}i_{\Lambda}$ preserves compact objects. This is because
  $i_{\Lambda}$ preserves compacts and $p_{X,*}$ is finite tor-amplitude and sends perfect objects
  supported on $\Lambda$ to almost perfect objects (see SAG 5.6.5.2).
\end{proof}

\begin{corollary} \label{crossshriek3}
  The map 
  \[\Hom(p_{X,*}\mathcal{F},\mathcal{G}) \to \Hom(\mathcal{F},p_X^!\mathcal{G})\]
  in Theorem \ref{crossshriek} is an isomorphism if $\mathcal{F}$ is supported on a proper (over $S$)
  subscheme.
\end{corollary}

Now we can prove the following alternative description of the $D_{X/S}$, also known as the Grothendieck-Sato formula.
\begin{corollary} \label{gsfor}
  \[D_{X/S} \cong \Gamma_{\Delta}(\pi_2^{*}(\omega_X)) \cong \Gamma_{\Delta}(\mathcal{O}_X \boxtimes \omega_X)\]
\end{corollary}
\begin{proof}
  By the theorem above we have the isomorphism
  \[\Gamma_{\Delta}\pi_1^{\times}(\mathcal{O}_X) \cong \Gamma_{\Delta}\pi_1^!(\mathcal{O}_X)\]
  By base-change for upper shriek (Theorem \ref{shriekbase}) we have the desired result.
\end{proof}

We conclude by showing that upper shriek composes well.
\begin{theorem} \label{csk3}
  Suppose $g: Y \to X$ is finite tor amplitude (as well as the standing
  assumptions). Then
  \[g^!p_X^! \cong p_Y^!\]
\end{theorem}
\begin{proof}
  Consider the diagram
  \[
    \begin{tikzcd}
      X \times_S Y  \arrow{r}{\pi_2} \arrow{dd}{\pi_1} & Y \arrow{dd}{p_Y} \\ \\
      X \arrow{r}{p_X} & S\\
    \end{tikzcd}
  \]
  We have
  \begin{equation*}
    \begin{split}
      p_Y^!
      &\cong \tilde{\delta}_Y^*\tilde{\pi}_2^{(Y \times Y),\times}p_Y^*\\
      &\cong \tilde{\delta}_Y^*\widetilde{(g \times \id)}^{\times}\tilde{\pi}_2^{(X \times Y),\times}p_Y^*\\
      &\cong \tilde{\delta}_Y^*\widetilde{(g \times \id)}^{\times}\tilde{\pi}_2^{(X \times Y),!}p_Y^*\\
      &\cong \tilde{\delta}_Y^*\widetilde{(g \times \id)}^{\times}\tilde{\pi}_1^{(X \times Y),*}p_X^!\\
    \end{split}
  \end{equation*}
  where $\tilde{\pi}_2^{(X \times Y),!} := \Gamma_Y\pi_2^{(X \times Y),!}$ and similarly for the $\tilde{\pi}_2^{(X \times Y),\times}$. The last isomorphism follows from Theorem \ref{shriekbase}.
  Now look at the Cartesian diagram
  \[
    \begin{tikzcd}
      Y \times_X Y  \arrow{rr}{\phi = \id \times_{p_X} \id} \arrow{dd}{\pi_2} && Y \times_S Y \arrow{dd}{g \times \id} \\ \\
      Y \arrow{rr}{\delta_g} && X \times_S Y \\
    \end{tikzcd}
  \]
  So,
  \begin{equation*}
    \begin{split}
      p_Y^!
      &\cong \tilde{\delta}_Y^*\widetilde{(g \times \id)}^{\times}\tilde{\pi}_1^{(X \times Y),*}p_X^!\\
      &\cong \tilde{\delta}_Y^*\phi^*\widetilde{(g \times \id)}^{\times}\tilde{\pi}_1^{(X \times Y),*}p_X^!\\
      &\cong \tilde{\delta}_Y^*\tilde{\pi}_2^{(Y \times_X Y),\times} \delta_g^*\tilde{\pi}_1^{(X \times Y),*}p_X^!\\
      &\cong \tilde{\delta}_Y^*\tilde{\pi}_2^{(Y \times_X Y),\times} g^*p_X^!\\
      &\cong g^!p_X^!\\
    \end{split}
  \end{equation*}
\end{proof}

\begin{remark}
  The statements of this sections indicates that Grothendieck duality, in the sense of constructing an upper shriek
  functor satisfies section 2 of \cite{N1}, can be developed from scratch using Definition \ref{defshriek}
  by making ample use of the category $\Gamma_{\Delta}(\QC{X \times X})$. In the section on smooth varieties, we
  follow Neeman and show that we can easily identify $\omega_X$ with the sheaf of top differential forms (shifted
  appropriatedly) in the smooth case.
\end{remark}

\begin{remark}
  Bhargav Bhatt pointed out that the upper shriek functor is also characterized (up to isomorphism) by the following properties. 

  \begin{enumerate}
  \item There is a map $p_X^{\times} \to p_X^!$ such that the induced map
    \[\Hom(p_{X,*}\mathcal{F},\mathcal{G}) \to \Hom(\mathcal{F},p_X^!\mathcal{G})\] is an isomorphism when $\mathcal{F}$
    has proper support over $S$ (this is the second half of Corollary \ref{crossshriek3}).
  \item Theorem \ref{shriekbase} holds.
  \end{enumerate}

  This observation can be deduced from the following diagram
  \begin{equation*}
    \begin{tikzcd}
        X \arrow{dr}{\delta} \\
        & X \times_S X  \arrow{r}{\pi_2} \arrow{dd}{\pi_1} & X \arrow{dd}{p_X} \\ \\
        & X \arrow{r}{p_X} & S\\
    \end{tikzcd}
    \end{equation*}
  Condition (2) implies that
  \[p_X^! \cong \delta^*\pi_2^*p_X^! \cong \delta^*\pi_1^!p_X^*\]
  Now condition (1) implies
  \[\id \cong \delta^{\times}\pi_1^{\times} \to \delta^{\times}\pi_1^!\]
  is an isomorphism. Hence
  \[\delta^*\pi_1^!p_X^* \cong \delta^*\pi_1^{\times}p_X^*\]
  using Lemma \ref{conserve}.
\end{remark}

\begin{remark}
  We work with usual underived schemes in this remark. Suresh Nayak pointed out to us that it is possible to define the upper shriek functor along arbitary maps
  of finitely presented separated schemes which are finite tor-amplitude over a Noetherian base by factoring such a map
  \[f : X \to Y\]
  as the composition of the graph of $f$
  \[\Gamma_f : X \to X \times Y\]
  composed with the projection map
  \[\pi_Y : X \times Y \to Y\]
  Then, we can define $f^!$ by the composition $\Gamma_f^{\times}\pi_Y^!$
  where $\pi_Y$ is finite tor-amplitude and hence we can define upper shriek along it
  using the techniques in this paper. However, for such a definition to be compatible with compositions,
  we must restrict to the subcategory $D^+_{qc}$ of objects with bounded below cohomology. However,
  we do not currently know how to adapt the category-theoretic proofs in this paper to this setting.
\end{remark}

\section{The Two Dualities}
In this section, we introduce two categorical dualities which are central to our paper. The first duality
interchanges a dualizable object in a symmetric monoidal category with its dual, which we refer to as up-down duality. The second
duality interchanges a dualizable category with its dual (inside $\Mod{\mathscr{V}}^L$ for some $\mathscr{V}$), which we refer to as left-right duality.
Up-down duality allows us to conjugate compact object preserving functors with taking duals of compact objects 
to obtain new functors (when compact objects coincide with dualizable objects). Left-right duality produces from a colimit-preserving functor between dualizable categories
a functor in the opposite direction on their duals. We will often be in a situation where our categories are in fact self-dual, where left-right duality
produces simply a functor in the reverse direction. For up-down duality, we quote extensively from \cite{BAS}.

Let us start with up-down duality. For a compactly generated presentable stable category $\mathscr{X}$, we denote by
$\mathscr{X}^c$ the stable subcategory of compact objects. Similarly, if $f$ is a colimit
preserving functor between compactly generated presentable stable categories which preserves compact objects,
we let $f^c$ be the functor restricted to compact objects.
Suppose $f : \mathscr{X} \to \mathscr{Y}$ is a map of compactly generated stable categories with a anti-automorphism on the compact
objects which preserves compact objects (for us this will always just be taking the dual of the object in a symmetric monoidal
category, which we can do because our compact objects are dualizable).
Then, we can conjugate the functor $f^c$ by the anti-automorphism to get a functor
\[(f^c)^D : \opposite{(\mathscr{X}^c)} \to \opposite{(\mathscr{Y}^c)}\]
By viewing $(f^c)^D$ as a functor from $\mathscr{X}^c$ to $\mathscr{Y}^c$, we can extend it uniquely to a colimit preserving functor
\[f^D : \mathscr{X} \to \mathscr{Y}\]

We record two lemmas paraphrased from \cite{BAS} (Lemma 2.6 in loc. cit)
\begin{lemma}
  Suppose $f : \mathscr{X} \to \mathscr{Y} $ is a colimit-preserving functor of compactly generated presentable stable
  categories which preserves compact objects. Then, $f^c : \mathscr{X}^c \to \mathscr{Y}^c$ has a right
  adjoint if and only if the right adjoint of $f$ preserves compact objects. In which case the right adjoint
  of $f$ is induced by the right adjoint of $f^c$.
\end{lemma}

\begin{lemma}
  Suppose $f : \mathscr{X} \to \mathscr{Y} $ is a colimit-preserving functor of compactly generated presentable stable
  categories which preserves compact objects. Then, $f^c : \mathscr{X}^c \to \mathscr{Y}^c$ has a left 
  adjoint if and only if $f$ has a left adjoint. In which case the left adjoint of $f$ is induced by
  the left adjoint of $f^c$.
\end{lemma}

We also record the following proposition from \cite{BAS}
\begin{proposition}
  Suppose $f : \mathscr{X} \to \mathscr{Y} $ is a colimit-preserving functor of compactly generated presentable stable
  categories (with anti-automorphisms as above) which preserves compact objects and such that $f^D \cong f$. Let $g$ be the right adjoint of $f$.
  Then, $f$ preserves limits if and only if $g$ preserves compact objects.
\end{proposition}
\begin{proof}
  We know that $f$ preserves limits if and only if $f$ has a left adjoint.
  By the second lemma above, $f$ has a left adjoint if and only if $f^c$ has a left adjoint.
  Now, $f^{c}$ has a left adjoint if and only if it has a right adjoint because it is invariant under duality.
  Finally, $f^{c}$ has a right adjoint if and only if $g$ preserves compact objects by the first lemma above.
\end{proof}

As a consequence, we have the following lemmas, which can be proven directly.
\begin{lemma}
  Suppose $f : \mathscr{X} \to \mathscr{Y}$ is a map of compactly generated stable categories with a anti-automorphism on the compact
  objects which preserves compact objects.
  Let $g$ be the right adjoint of $f$ and suppose $g$ preserves compact objects.
  Then, $g^D$ is the left adjoint of $f^D$.
\end{lemma}

\begin{lemma}
  Suppose $f : \mathscr{X} \to \mathscr{Y}$ is a map of compactly generated stable categories with a anti-automorphism on the compact
  objects which preserves compact objects and limits.
  Let $g$ be the left adjoint of $f$.
  Then, $g^D$ is the right adjoint of $f^D$.
\end{lemma}

Now let us discuss left-right duality. Suppose $\mathscr{X}$ and $\mathscr{Y}$ are dualizable $\mathscr{V}$-categories, in the notation
of Appendix \ref{catresult}. Then for $f : \mathscr{X} \to \mathscr{Y}$ a colimit-preserving $\mathscr{V}$-linear functor, there is a colimit preserving dual functor
\[f^{\vee} : \mathscr{Y}^{\vee} \to \mathscr{X}^{\vee} \]
We refer to this duality as left-right duality. 
Left-right duality also interchanges adjunctions, namely the following is easily seen.
\begin{proposition}
  Suppose $f: \mathscr{X} \to \mathscr{Y}$ is left adjoint to $g : \mathscr{Y} \to \mathscr{X}$ and both are colimit-preserving
  $\mathscr{V}$-linear functors between $\mathscr{V}$-dualizable categories, then $g^{\vee}$ is left adjoint to $f^{\vee}$.
\end{proposition}
\begin{corollary} \label{compactlimit}
  Suppose $f: \mathscr{X} \to \mathscr{Y}$ is a $\mathscr{V}$-linear colimit-preserving functor between compactly generated
  $\mathscr{V}$-module categories. Then $f$ preserves compacts if and only if
  $f^{\vee}$ is limit preserving.
\end{corollary}
\begin{proof}
  $f$ preserves compact objects if and only if it has a colimit-preserving right adjoint, which is true if and only if $f^{\vee}$ has
  a left adjoint, which is equivalent to $f^{\vee}$ preserving limits.
\end{proof}

Left-right duality does not change the kernels of Fourier-Mukai transforms. More precisely, the following is also easily checked
\begin{proposition} \label{kernelp}
  Suppose \[f: \mathscr{X} \to \mathscr{Y}\] is given by the Fourier-Mukai transform with kernel \[K \in \mathscr{X}^{\vee} \otimes_{\mathscr{V}} \mathscr{Y}\]
  (all colimit-preserving $\mathscr{V}$-linear functors are of this form) then
  \[f^{\vee} : \mathscr{Y}^{\vee} \to \mathscr{X}^{\vee}\] is given by the same kernel $K$ inside
  \[(\mathscr{Y}^{\vee})^{\vee} \otimes_{\mathscr{V}} \mathscr{X}^{\vee} \cong \mathscr{X}^{\vee} \otimes_{\mathscr{V}} \mathscr{Y}\]
\end{proposition}
\begin{remark}
  Suppose $\mathscr{V} \cong \Mod{k}$ for a commutative ring $k$, $\mathscr{X} \cong \Mod{A}$, and $\mathscr{Y} \cong \Mod{B}$
  for some $k$-algebras $A$ and $B$. Let $f : \Mod{A} \to \Mod{B}$ be given by tensoring over $A$ with some $(B,A)$ bimodule $M$. In this case
  $f^{\vee}: \Mod{\opposite{B}} \to \Mod{\opposite{A}}$ is given by tensoring over $\opposite{B}$ with the same $M$, thought of as a $(\opposite{A},\opposite{B})$ bimodule.
\end{remark}
In practice we will almost never use the the superscript $^{\vee}$ to denote left-right duality. We note here that if $X$ is a spectral scheme over $S$, satisfying the conditions of the previous sections,
$\QC{X}$ is always self-dual over $\QC{S}$ (see \cite{SAG} 9.4.2.2 and 9.4.3.2). As a consequence of the Proposition \ref{kernelp}, we note that
left-right duality switches quasicoherent pullback with quasicoherent pushforward, as they are given by the same Fourier-Mukai kernels. Finally, suppose
we are given Noetherian spectral schemes $X$ over $S$ as in the Section \ref{G}. Let $Z$ be a closed subscheme of $X$. Then
\begin{proposition} \label{selfdual}
  $\Gamma_Z(\QC{X})$ is self-dual and left-right duality interchanges $i_Z$ with $\Gamma_Z$.
\end{proposition}
\begin{proof}
  We can apply the same argument as the standard proof that $\QC{X}$ is self-dual when $X$ is a perfect stack (for example Corollary 4.8 in \cite{BZFN}, though
  note that they use a stronger than necessary definition of perfect stack).
  The only difference is that when showing $\QC{X}$ is self-dual, the unit and counit maps are given Fourier-Mukai transforms with the kernel
  \[\mathscr{O}_{\Delta} \in \QC{X \times_S X}\]
  Whereas to show $\Gamma_Z(\QC{X})$ is self-dual, we use instead the kernel
  \[\Gamma_Z(\mathcal{O}_{\Delta}) \in \Gamma_Z(\QC{X}) \otimes_{\QC{S}} \Gamma_Z(\QC{X})\]
  The rest of the proof proceeds the same way as in \cite{BZFN}.

  For the second part of the proposition, simply check that both functors are given by the same Fourier-Mukai kernel, namely,
  \[\Gamma_Z(\mathcal{O}_{\Delta}) \in \Gamma_Z(\QC{X}) \otimes_{\QC{S}} \QC{X}\]
\end{proof}

\section{Dualizing Complexes and the Lower Shriek Functor}
In this section, we introduce the lower shriek functor and prove some Hochschild-type formulas
which appear in \cite{N1}. We have seen that the upper shriek functor satisfies
\[p^!(\_) \cong p^*(\_) \otimes \omega\]
Now the lower shriek functor will turn out to satisfy an analogous equation, namely,
\[p_!(\_) \cong p_*(\_ \otimes \omega)\]
In fact, these two are simply related by left-right duality. 
We also caution that our use of the symbol lower shriek is not necessarily standard, in particular
it is not analogous to the étale lower shriek. However, this notation is not original either, for example see \cite{Perry}.
We insist on this notation because it is consistent with how the rest of our notation behaves under left-right duality.
Much of this section is inspired by arguments in \cite{BAS} and \cite{N1}.

Suppose $p_X :X \to S$ satisfies the hypotheses in Section \ref{G}. The
following theorem is the left-right dual of Theorem \ref{AGinput1}.
\begin{theodef} \label{AGcoro2}
  Denote by $\tilde{\pi}_1^*$ the functor
  \[\Gamma_{\Delta}\pi_1^* : \QC{X} \to \Gamma_{\Delta}(\QC{X \times_S X})\]
  Then, $\tilde{\pi}_1^*$ preserves limits--we denote by $\tilde{\pi}_{1,\times}$ its
  left adjoint. 
\end{theodef}
\begin{proof}
  $\tilde{\pi}_1^*$ is the left-right dual of $\tilde{\pi}_{1,*}$, so the theorem
  follows from Corollary \ref{compactlimit} applied to Theorem \ref{AGinput1}.
\end{proof}

\begin{remark}
  We note that $\tilde{\pi}_{1,\times}$ preserves compact objects because $\tilde{\pi}_1^*$ is colimit-preserving.
  Also, $\tilde{\pi}_{1,\times}$ is left-right dual to $\tilde{\pi}_1^{\times}$.
\end{remark}

Let $\delta : X \to X \times_S X$ be the diagonal map. Then,
\[\delta_*: \QC{X} \to \QC{X \times_S X}\]
factors through
\[i_{\Delta} : \Gamma_{\Delta}(\QC{X \times_S X}) \to \QC{X \times_S X}\]
Namely,
\[\delta_* \cong i_{\Delta} \tilde{\delta}_*\]
We are now ready to define the lower shriek functor, in analogy to the upper shriek functor.
\begin{definition}
  The lower shriek functor $p_{X,!}: \QC{X} \to \QC{S}$ is defined by
  \[p_{X,!}:= p_{X,*}\tilde{\pi}_{1,\times}\tilde{\delta}_*\]
\end{definition}
\begin{remark}
  By comparison with (\ref{eq2}) it is clear that $p_{X,!}$ is the left-right dual of $p_X^!$.
\end{remark}

We can now take most of the results of section 2 and apply left-right duality to them to
obtains new results about lower shriek. For example, we have the following analogue of Proposition \ref{openshriek}, which follows directly from left-right duality.
\begin{proposition}
  Suppose $u : U \to X$ is an étale map , then
  \[p_{U,!} \cong p_{X,!}u_* \cong p_{X,!}u_!\]
\end{proposition}

Also, we can take the left-right dual of (\ref{shriekeq}) to get
\begin{proposition}
  \[p_{X,!}(\mathcal{F}) \cong p_{X,*}(\mathcal{F} \otimes \omega_X)\]
\end{proposition}
\begin{remark}
  We can also show these results directly by arguing with compact objects, however we choose to present the proofs by duality because
  they are cleaner.
\end{remark}

As a preparation for the next theorem, we need the following result.
\begin{proposition}
  For $\mathcal{F} \in \QC{X}$,
  \[\tilde{\pi}_{1,\times}\tilde{\delta}_*\mathcal{F} \cong \tilde{\delta}^* \tilde{\pi}_1^{\times}\mathcal{F} \cong \mathcal{F} \otimes \omega_{X/S}\]
  as $\QC{X}$-linear colimit-preserving functors.
\end{proposition}
\begin{proof}
  Both the first and second expression are $\QC{X}$-linear colimit preserving functors of $\mathcal{F}$ by Corollary \ref{AGcoro1} and Corollary \ref{AGcoro2}.
  $\QC{X}$-linear colimit preserving functors from $\QC{X}$ to itself are automatically self-dual because they are simply given
  by tensoring with a quasicoherent sheaf on $X$, showing the first equality. In this case, it is easy to see the functor
  is given by tensoring with $\omega_{X/S}$. This shows the claim.
\end{proof}

We are now ready to establish a Hochschild-style formula which is known in some form since \cite{AILN} and is elaborated on in \cite{N1}.
\begin{theorem} \label{hochschildstyle}
  For $\mathcal{F} \in \QC{S}$ and $\mathcal{G} \in \QC{X}$, we have
  \[\delta^{\times}\pi_1^*\iHom(p_X^*\mathcal{F},\mathcal{G}) \cong \iHom(p_X^!\mathcal{F},\mathcal{G})\]
  where $\iHom$ denotes internal Hom of quasicoherent sheaves.
\end{theorem}
\begin{proof}
  Consider
  \[\delta^{\times}: \QC{X \times_S X} \to \QC{X}\]
  which is right adjoint to
  \[\delta_* \cong i_{\Delta}\Gamma_{\Delta}\delta_* \cong i_{\Delta}\tilde{\delta}_*\]
  Hence,
  \[\delta^{\times} \cong \tilde{\delta}^{\times}\Gamma_{\Delta}\]
  where $\tilde{\delta}^{\times}$ is right adjoint to $\tilde{\delta}_*$.

  Therefore, given $\mathcal{H} \in \QC{X}$, we have
  \begin{equation*}
    \begin{split}
      Hom_X(\mathcal{H},\delta^{\times}\pi_1^*\iHom(p_X^*\mathcal{F},\mathcal{G})) 
      &\cong Hom_X(\mathcal{H},\tilde{\delta}^{\times}\tilde{\pi}_1^*\iHom(p_X^*\mathcal{F},\mathcal{G})) \\
      &\cong Hom_X(\tilde{\pi}_{1,\times}\tilde{\delta}_*\mathcal{H},\iHom(p_X^*\mathcal{F},\mathcal{G})) \\
      &\cong Hom_X(\omega_{X/S} \otimes \mathcal{H},\iHom(p_X^*\mathcal{F},\mathcal{G})) \\
      &\cong Hom_X(\omega_{X/S} \otimes p_X^*\mathcal{F} \otimes \mathcal{H},\mathcal{G}) \\
      &\cong Hom_X(p_X^!\mathcal{F} \otimes \mathcal{H},\mathcal{G}) \\
      &\cong Hom_X(\mathcal{H},\iHom(p_X^!\mathcal{F},\mathcal{G})) \\
    \end{split}
  \end{equation*}
\end{proof}

Notice that if $X = \Spec{A}$ and $S = \Spec{k}$, then this theorem says (in a special case)
\[\Hom_{A \otimes A}(A,A \otimes A) \cong \Hom_A(\omega_A,A)\]
\begin{corollary}
  \[\iHom(\omega_X,\omega_X) \cong \mathcal{O}_X\]
\end{corollary}
\begin{proof}
  \begin{equation*}
    \begin{split}
      \iHom(\omega_X,\omega_X)
      &\cong \tilde{\delta}^{\times}\tilde{\pi}_2^*\omega_X \\
      &\cong \tilde{\delta}^{\times}D_X \\
      &\cong \tilde{\delta}^{\times}\tilde{\pi}_1^{\times}\mathcal{O}_X \\
      &\cong \mathcal{O}_X \\
    \end{split}
  \end{equation*}
  The first isomorphism comes from the theorem above and the second is the Grothendieck-Sato formula.
\end{proof}

We record here a proposition which is morally dual to Proposition \ref{crossshriek} and Corollary
\ref{crossshriek3}, though we don't know how to show it directly by duality.
\begin{proposition}
  There is a natural map
  \[Hom(p_{X,!}\mathcal{F},\mathcal{G}) \to Hom(\mathcal{F},p_X^*\mathcal{G}) \]
  which is an isomorphism if the support of $\mathcal{F}$ is proper over $S$.
\end{proposition}
\begin{proof}
  The map is constructed as follows
  \begin{equation*}
    \begin{split}
      Hom(p_{X,!}\mathcal{F},\mathcal{G})
      &\cong Hom(p_{X,*}\tilde{\pi}_{1,\times}\tilde{\delta}_*\mathcal{F},\mathcal{G}) \\
      &\cong Hom(\tilde{\delta}_*\mathcal{F},\tilde{\pi}_1^*p_X^{\times}\mathcal{G}) \\
      &\to Hom(\tilde{\delta}_*\mathcal{F},\tilde{\pi}_1^*p_X^!\mathcal{G}) \\
      &\cong Hom(\tilde{\delta}_*\mathcal{F},\Gamma_{\Delta}\pi_2^!p_X^*\mathcal{G}) \\
      &\cong Hom(\mathcal{F},\tilde{\delta}^{\times}\Gamma_{\Delta}\pi_2^!p_X^*\mathcal{G}) \\
      &\cong Hom(\mathcal{F},\tilde{\delta}^{\times}\Gamma_{\Delta}\pi_2^{\times}p_X^*\mathcal{G}) \\
      &\cong Hom(\mathcal{F},\tilde{\delta}^{\times}\tilde{\pi}_2^{\times}p_X^*\mathcal{G}) \\
      &\cong Hom(\mathcal{F},p_X^*\mathcal{G}) \\
    \end{split}
  \end{equation*}
  where the map in the third line comes from Proposition \ref{crossshriek}.
  The fourth line is base-change for upper shriek (see Theorem \ref{shriekbase}).
  On the sixth line we apply Theorem \ref{crossshriek4}.

  If $Z$ is the support of $\mathcal{F}$, then assuming $Z$ is proper over $S$, we want to show that the
  map on line three is an isomorphism. Indeed, 
  \begin{equation*}
    \begin{split}
      Hom(\tilde{\delta}_*\mathcal{F},\tilde{\pi}_1^*p_X^{\times}\mathcal{G}) 
      &\cong Hom(\tilde{\delta}_*\mathcal{F},\Gamma_{Z \times Z}\Gamma_{\Delta}\pi_{1}^*p_X^{\times}\mathcal{G}) \\
      &\cong Hom(\tilde{\delta}_*\mathcal{F},\Gamma_{\Delta}\Gamma_{Z \times Z}\pi_{1}^*p_X^{\times}\mathcal{G}) \\
      &\cong Hom(\tilde{\delta}_*\mathcal{F},\Gamma_{\Delta}\pi_{1}^*\Gamma_Zp_X^{\times}\mathcal{G}) \\
      &\cong Hom(\tilde{\delta}_*\mathcal{F},\Gamma_{\Delta}\Gamma_{X \times Z}\pi_2^{\times}p_X^*\mathcal{G}) \\
      &\cong Hom(\tilde{\delta}_*\mathcal{F},\Gamma_{Z \times Z}\Gamma_{\Delta}\pi_2^{\times}p_X^*\mathcal{G}) \\
      &\cong Hom(\tilde{\delta}_*\mathcal{F},\Gamma_{\Delta}\pi_2^{\times}p_X^*\mathcal{G}) \\
    \end{split}
  \end{equation*}
  where the fourth isomorphism follows from Theorem \ref{BC2} applied to
  $\mathscr{V} = \QC{S}$, $\mathscr{X} = \Gamma_Z(\QC{X})$, and $\mathscr{Y} = \QC{X}$, where the map
  $f = \Gamma_Zp_X^{\times} : \mathscr{V} \to \mathscr{X}$
  is the right adjoint of
  \[p_{X,*}i_Z : \Gamma_Z(\QC{X}) \to \QC{S}\]
  $f$ is colimit-preserving because $p_{X,*}i_Z$ preserves compact objects (argue as in Theorem \ref{AGinput1}).
  The map $g : \mathscr{V} \to \mathscr{Y}$
  is just the quasicoherent pullback.
\end{proof}

Lastly, we record a theorem about how our functors interact with up-down duality.
\begin{theorem}
  \[(\tilde{\pi}_{1,\times})^D \cong \tilde{\pi}_{1,*}\]
 and the isomorphism is étale local.
\end{theorem}
\begin{proof}
  It suffices to show there is a natural isomorphism
  \[\Hom_{\QC{X}}(\tilde{\pi}_{1,\times}(K^{\vee}),L) \cong \Hom_{\QC{X}}((\tilde{\pi}_{1,*}K)^{\vee},L)\]
  for $K$ compact in $\Gamma_{\Delta}(\QC{X \times_S X})$ and $L$ in $\QC{X}$.
  But this follows from
  \begin{equation*}
    \begin{split}
      \Hom_{\QC{X}}(\tilde{\pi}_{1,\times}(K^{\vee}),L)
      &\cong \Hom_{\Gamma_{\Delta}(\QC{X \times X})}(K^{\vee},\tilde{\pi}_1^*L) \\
      &\cong \Hom_{\Gamma_{\Delta}(\QC{X \times X})}(K^{\vee},\Gamma_{\Delta}\pi_1^*L) \\
      &\cong \Hom_{\QC{X \times X}}(K^{\vee},\pi_1^*L) \\
      &\cong \Hom_{\QC{X \times X}}(\mathcal{O}_{X \times X},K \otimes_{\mathcal{O}_{X \times X}}\pi_1^*L) \\
      &\cong \Hom_{\QC{X}}(\mathcal{O}_X,\pi_{1,*}(K \otimes_{\mathcal{O}_{X \times X}} \pi_1^*L)) \\
      &\cong \Hom_{\QC{X}}(\mathcal{O}_X,\pi_{1,*}K \otimes_{\mathcal{O}_X} L) \\
      &\cong \Hom_{\QC{X}}((\pi_{1,*}K)^{\vee},L) \\
      &\cong \Hom_{\QC{X}}((\tilde{\pi}_{1,*}K)^{\vee},L) \\
    \end{split}
  \end{equation*}
  where by abuse of notation $K$ can also be thought of as an object in $\QC{X \times_S X}$.
\end{proof}

\section{Comparison with Classical Definitions for Smooth Varieties}
In this section only, assume $S$ is a discrete (i.e. non-derived) Noetherian scheme, and $X$ is a smooth variety over $S$.
In this case, Grothendieck defined the ring (sheaf of rings) of Grothendieck differential operators on $X$ relative to $S$ (\cite{EGA4}).
We will show in this section that our definition agrees with this standard definition in this case. Moreover,
we will show that the dualizing complex is given by the sheaf of top differential forms homologically shifted
by the dimension of the variety, following Neeman \cite{N2}. Taken together, this yields a simple and powerful method for deducing Serre
duality from scratch.

We begin by showing that the ring of Grothendieck differential operators classically defined agrees with our definition. Strictly speaking,
we can't quite do that yet as we haven't said what the ring structure is in our case yet. However, we will see that at least locally, there is
an obvious choice, which agrees with the classical definition.
The following theorem is known, for example see \cite{SVdB}, but we provide a proof here as well.
\begin{theorem} \label{smoothagree}
  In the case of $X$ a smooth variety over a discrete Noetherian scheme $S$, our definition of $D_X$ agrees with the classical definition of Grothendieck
  differential operators (and hence is discrete).
\end{theorem}
\begin{proof}
  We will show that affine locally, there is a canonical isomorphism. This will then imply the global statement.

  Suppose we have $X \cong \Spec{R}$ smooth over $S \cong \Spec{k}$, both discrete Noetherian rings.
  Our definition in this case yields $D_{R/k} \cong \Gamma_{\Delta}(Hom_k(R,R))$ (see Definition \ref{affinediff}).
  We will take as the classical definition of the Grothendieck differential operators 
  \[\mathcal{D} := \bigcup_{n \ge 0} {D^{(n)}}\]
  the union of the increasing
  sequence of subspaces $D^{(n)} \subseteq \Hom_k(R,R)$ defined inductively by
  \[D^{(-1)}=0\]
  and
  \[D^{(n)}=\{f \in \Hom_k(R,R) | \forall r \in R,  [f,r] \in D^{(n-1)}\}\]
  where $r \in R$ is thought of as an element $\Hom_k(R,R)$ via multiplication by $r$.

  Now let $I$ be the ideal in $R \otimes_k R$ defining the diagonal. Recall that $\Hom_k(R,R)$
  has an action of $R \otimes_k R$ via
  \[((a_1 \otimes a_2)f)(x) = a_1f(a_2x)\]
  Therefore, the condition that \[\forall r \in R, [f,r] \in D^{(n-1)}\] is equivalent to
  \[\forall r \in R, (r \otimes 1)f - (1 \otimes r)f \in D^{(n-1)}\]
  which is further equivalent to
  \[If \in D^{(n-1)}\]
  Therefore, we can conclude that
  \[D^{(n)} \cong H^0\Hom_{R \otimes_k R}((R \otimes_k R)/I^n,\Hom_k(R,R))\]
  By adjunction we have
  \begin{equation*}
    \begin{split}
      \Hom_{R \otimes_k R}((R \otimes_k R)/I^n,\Hom_k(R,R))
      &\cong \Hom_{R \otimes_k R}((R \otimes_k R)/I^n,\Hom_R(R \otimes_k R,R)) \\
      &\cong \Hom_R((R \otimes R)/I^n,R)
    \end{split}
  \end{equation*}
  where the $R$ action is on the first factor of the tensor. However,
  because $R$ is smooth, we have a noncanonical isomorphism
  \[(R \otimes R)/I^n \cong \oplus_{i=0}^{n-1}(\Sym^k(\Omega_{R/k}))\]
  and hence
  \[D^{(n)} \cong \Hom_{R \otimes_k R}((R \otimes_k R)/I^n,\Hom_k(R,R))\]
  because the right hand side is concentrated in degree zero (since $(R \otimes_k R)/I^n$
  is projective).

  Therefore, as filtered colimits are exact
  \begin{equation*}
    \begin{split}
      \mathcal{D}
      &\cong \colim_n \Hom_{R \otimes_k R}((R \otimes_k R)/I^n,\Hom_k(R,R)) \\
      &\cong \Gamma_{\Delta}(\Hom_k(R,R)) \\
      &\cong D_{R/k}
    \end{split}
  \end{equation*}
\end{proof}
Now let us move on to verifying that our definition of the dualizing complex gives
the top differential forms in homological degree $n$ in the smooth setting.
The idea of the following proof is due to Lipman and is written in \cite{lipman}, it is also
presented in Section 3.2 of \cite{N2}.

The intermediate object connecting differential forms with $\omega_X$ is Hochschild homology, which can be written as
\[\HH(X/S) := \delta^*\delta_*\mathcal{O}_X \cong \tilde{\delta}^*\tilde{\delta}_*\mathcal{O}_X\]
Because of the isomorphism
\[\tilde{\pi}_{1,*}\tilde{\delta}_* \cong \id_{\QC{X}}\]
there is a natural map (by adjunction)
\[\tilde{\delta}_*\mathcal{O}_X \to \tilde{\pi}_1^{\times}\mathcal{O}_X\]
Therefore by applying $\tilde{\delta}^*$ on both sides there is a natural map
\[\HH(X/S) \to \omega_X\]
By the HKR isomorphism, in the smooth case, we also have a map
\[\Omega^n_{X/S}[n] \cong \pi_{\ge n}\HH(X/S) \to \HH(X/S)\]
where $\Omega^i_{X/S}$ is the sheaf of $i$-forms. These combine to form a natural map
\[\Omega^n_{X/S}[n] \to \omega_X\] in the smooth case. We wish to show it's an isomorphism.

By étale descent, it is enough to check it for $\mathbb{A}^n$.
Now, for $X,Y$ over $S$ satisfying our standing assumptions
\[\Gamma_{\Delta}(\QC{(X \times_S Y) \times_S (X \times_S Y)}) \cong \Gamma_{\Delta}(\QC{X \times_S X}) \otimes \Gamma_{\Delta}(\QC{Y \times_S Y})\]
and
\[\tilde{\pi}^{(X \times Y)}_{1,*} \cong \tilde{\pi}^{(X)}_{1,*} \boxtimes \tilde{\pi}^{(Y)}_{1,*}\]
therefore also
\[\tilde{\pi}_1^{\times, (X \times Y)} \cong \tilde{\pi}_1^{\times,(X)} \boxtimes \tilde{\pi}_1^{\times,(Y)}\]
Hence we have
\[D_{X \times Y} \cong D_X \boxtimes D_Y\]
Pulling back along the diagonal, we get
\[\omega_{X \times Y} \cong \omega_X \boxtimes \omega_Y\]
We also have similar results for Hochschild homology and $\Omega^n_{X/S}[n]$ compatible with the maps between
them. Therefore, it suffices to show the isomorphism for $\mathbb{A}^1$. So the result follows from

\begin{lemma}
  Over a Noetherian base scheme $S$ over $\Spec{\mathbb{Z}}$, 
  \[\omega_{\mathbb{A}^1/S} \cong \mathcal{O}_{\mathbb{A}^1/S}[1]\]
  and the map
  \[\mathcal{O}_{\mathbb{A}^1} \oplus \Omega_{\mathbb{A}^1}[1] \cong \HH(\mathbb{A}^1) \to \omega_{\mathbb{A}^1}\]
  is an isomorphism in degree $1$.
\end{lemma}
\begin{proof}
  By base-change results, we can assume $S \cong \Spec{\mathbb{Z}}$.
  We have (by Definition \ref{dualizedef})
  \[\omega_{\mathbb{Z}[x]/\mathbb{Z}} \cong D_{\mathbb{Z}[x]/\mathbb{Z}} \otimes_{\mathbb{Z}[x_1,x_2]} \mathbb{Z}[x]\]
  where the map
  \[\mathbb{Z}[x_1,x_2] \to \mathbb{Z}[x]\]
  sends $x_1$ and $x_2$ to $x$.
  $\mathbb{Z}[x]$ has the following resolution over $\mathbb{Z}[x_1,x_2]$.
  \[\mathbb{Z}[x_1,x_2] \xrightarrow{(x_1-x_2) \cdot \_} \mathbb{Z}[x_1,x_2] \to \mathbb{Z}[x]\]
  Hence by tensoring with $D_{\mathbb{Z}[x]/\mathbb{Z}}$, we have the following exact triangle in $\QC{\mathbb{Z}[x_1,x_2]}$.
  \[D_{\mathbb{Z}[x]} \xrightarrow{[x,\_]} D_{\mathbb{Z}[x]} \to \omega_{\mathbb{Z}[x]}\]
  where the first map is conjugating by multiplication by $x$.

  By a direct computation $D_{\mathbb{Z}[x]/\mathbb{Z}}$ is a free $\mathbb{Z}$ module
  on the generators $\{\frac{1}{n!}\frac{d}{dx}\}_{n \ge 0}$.
  Therefore the map
  \[[x,\_] : D_{\mathbb{Z}[x]/\mathbb{Z}} \to D_{\mathbb{Z}[x]/\mathbb{Z}}\] is surjective and
  the kernel is just $\mathbb{Z}[x]$. Therefore,
  \[\omega_{\mathbb{Z}[x]} \cong \mathbb{Z}[x][1]\]
  Now, we also have the triangle (by the same resolution of $\mathbb{Z}[x]$ above)
  \[\mathbb{Z}[x] \xrightarrow{[x,\_]} \mathbb{Z}[x] \to \HH(\mathbb{Z}[x])\]
  which naturally maps to the triangle
  \[D_{\mathbb{Z}[x]} \xrightarrow{[x,\_]} D_{\mathbb{Z}[x]} \to \omega_{\mathbb{Z}[x]}\]
  The lemma follows from direct calculation.
\end{proof}

We have therefore shown
\begin{theorem}
  For $p_X: X \to S$ a smooth map of relative dimension $n$
  (where $S$ is a discrete scheme), there is a natural isomorphism
  \[\omega_{X/S} \cong \Omega^n_{X/S}[n]\]
\end{theorem}

\appendix
\section{Background on Module Categories} \label{catresult}
In this section we collect some results in category theory which we will use throughout the document.

We denote by $Pr_{St}^L$ the $2$-category of presentable stable categories with colimit preserving functors.
By section 4.8.3 of \cite{HA}, there is a tensor product on $Pr_{St}^L$.
Therefore, let $\mathscr{V}$ be a monoidal presentable stable category, $\mathscr{X}$ a right $\mathscr{V}$ module
and $\mathscr{Y}$ a left $\mathscr{V}$ module (inside $Pr_{St}^L$). Then, using section
4.4 of \cite{HA}, we can form the relative tensor product of $\mathscr{X}$ and $\mathscr{Y}$ over $\mathscr{V}$, namely,
\[\mathscr{X} \otimes_{\mathscr{V}} \mathscr{Y}\]
We record two basic properties of this tensor product here for easy use later

\begin{theorem} \label{BC1}
  Suppose we have a functor $f : \mathscr{X} \to \mathscr{V}$ which is right $\mathscr{V}$-linear and colimit-preserving
  and $g: \mathscr{V} \to \mathscr{Y}$ which is left $\mathscr{V}$-linear and colimit-preserving.
  Then, the following diagram commutes
  \[
    \begin{tikzcd}
        \mathscr{X} \cong \mathscr{X} \otimes_{\mathscr{V}} \mathscr{V}  \arrow{r}{1 \otimes g} \arrow{dd}{f} & \mathscr{X} \otimes_{\mathscr{V}} \mathscr{Y} \arrow{dd}{f \otimes 1} \\ \\
        \mathscr{V} \arrow{r}{g} & \mathscr{Y} \cong \mathscr{V} \otimes_{\mathscr{V}} \mathscr{Y}\\
    \end{tikzcd}
  \]
\end{theorem}
\begin{proof}
  Follows from functoriality of the relative tensor product.
\end{proof}

\begin{remark}
  We must warn the reader that when we write a functor is $\mathscr{V}$-linear, we mean that it is a map of $\mathscr{V}$ module categories
  as defined above. Under the equivalence in \cite{MGS} A.3.8,
  a $\mathscr{V}$-module category can also be thought of as a presentable $\mathscr{V}$-category in the terminology of \textit{loc. cit.}.
  However, in the language of enriched category theory, the corresponding notion is referred to as preserving $\mathscr{V}$-tensors, and the terminology
  ``$\mathscr{V}$-linear'' is often used to refer to a weaker condition.
\end{remark}

\begin{remark}
  We can think of the above theorem as a category-theoretic analogue of the base-change isomorphism in algebraic geometry.
\end{remark}

\begin{theorem} \label{BC2}
  Suppose $f : \mathscr{V} \to \mathscr{X}$ is right $\mathscr{V}$-linear and colimit-preserving
  and $g: \mathscr{V} \to \mathscr{Y}$ is left $\mathscr{V}$-linear and colimit-preserving.
  Then, the following diagram commutes
  \[
    \begin{tikzcd}
        \mathscr{V} \arrow{r}{g} \arrow{dd}{f} & \mathscr{Y} \cong \mathscr{V} \otimes_{\mathscr{V}} \mathscr{Y} \arrow{dd}{f \otimes 1} \\ \\
        \mathscr{X} \cong \mathscr{X} \otimes_{\mathscr{V}} \mathscr{V} \arrow{r}{1 \otimes g} & \mathscr{X} \otimes_{\mathscr{V}} \mathscr{Y}\\
    \end{tikzcd}
  \]
\end{theorem}
\begin{proof}
  Follows from functoriality of the relative tensor product.
\end{proof}

\begin{theorem} \label{BC3}
  Suppose $f : \mathscr{X} \to \mathscr{V}$ is right $\mathscr{V}$-linear and colimit-preserving
  and $g: \mathscr{Y} \to \mathscr{V}$ is left $\mathscr{V}$-linear and colimit-preserving.
  Then, the following diagram commutes
  \[
    \begin{tikzcd}
        \mathscr{X} \otimes_{\mathscr{V}} \mathscr{Y} \arrow{r}{f \otimes 1} \arrow{dd}{1 \otimes g} & \mathscr{Y} \cong \mathscr{V} \otimes_{\mathscr{V}} \mathscr{Y} \arrow{dd}{g} \\ \\
        \mathscr{X} \cong \mathscr{X} \otimes_{\mathscr{V}} \mathscr{V} \arrow{r}{f} & \mathscr{V}\\
    \end{tikzcd}
  \]
\end{theorem}
\begin{proof}
  Follows from functoriality of the relative tensor product.
\end{proof}

Now let $\mathscr{V}$ be a symmetric monoidal compactly generated stable category,
such that the compact objects are the same as the dualizable objects. This will happen
whenever $\mathscr{V}$ is the category of quasicoherent sheaves on a (qcqs) spectral scheme
by Proposition 6.2.6.2 of \cite{SAG}. Let us consider
the $2$-category of presentable stable category $\mathscr{X}$ with a colimit-preserving left action of $\mathscr{V}$
where the morphisms are colimit preserving functors which preserve the $\mathscr{V}$ action. We denote this
$2$-category by $\Mod{\mathscr{V}}^L$ and we call its objects $\mathscr{V}$-module categories. In this setting,
there are enriched forms of adjoint functor theorems.

\begin{theorem} \label{Proj1}
  Let $\mathscr{X}$ and $\mathscr{Y}$ be compactly generated $\mathscr{V}$-module categories.
  Suppose $f : \mathscr{X} \to \mathscr{Y}$ is a colimit-preserving functor between $\mathscr{V}$-module categories which preserves compact objects,
  then the right adjoint $g$ can be upgraded to a $\mathscr{V}$-linear functor.
  (Please note that the assumptions on $\mathscr{V}$ in this theorem are stronger than the beginning of the section, see the previous paragraph)
\end{theorem}
\begin{proof}
  See A.3.6 in \cite{MGS}
\end{proof}

\begin{remark}
  We can think of the above as a category-theoretic analogue of the projection formula in algebraic geometry.
\end{remark}

\begin{theorem} \label{Proj2}
  Suppose $g : \mathscr{X} \to \mathscr{Y}$ is a colimit-preserving functor between $\mathscr{V}$-module categories which preserves limits,
  then the left adjoint $f$ can be upgraded to a $\mathscr{V}$-linear functor.
  (Please note that the assumptions on $\mathscr{V}$ in this theorem are stronger than the beginning of the section, see the paragraph before Theorem \ref{Proj1})

\end{theorem}
\begin{proof}
  See A.3.6 in \cite{MGS}
\end{proof}

\section{Exact Sequences of Categories}
In this section, we record what it means for a sequence of categories to be exact or split-exact. We use \cite{BGT} as our reference.
Note that our definition for split-exactness differs from theirs.

\begin{definition}[Definition 5.4 of \cite{BGT}]
  Let $f : \mathscr{A} \to \mathscr{B}$ be a fully faithful functor of presentable stable
  categories (this implies that $f$ preserves colimits). The Verdier quotient $\mathscr{B}/\mathscr{A}$ of $\mathscr{B}$
  by $\mathscr{A}$ is the cofiber of $f$ in the category $Pr_{St}^L$ of presentable stable categories.
\end{definition}

\begin{definition}[Definition 5.8 of \cite{BGT}]
  A sequence of presentable stable categories
  \[\mathscr{A} \to \mathscr{B} \to \mathscr{C}\]
  is exact if the composite is trivial, $\mathscr{A} \to \mathscr{B}$ is fully faithful, and the map $\mathscr{B}/\mathscr{A} \to \mathscr{C}$ is an equivalence.
\end{definition}

\begin{definition} \label{splitexact}
  An exact sequence of presentable stable categories
  \[\mathscr{A} \xrightarrow{f} \mathscr{B} \xrightarrow{g} \mathscr{C}\]
  is split-exact if the are colimit-preserving right adjoints $i$ and $j$ (to $f$ and $g$ respectively)
  such that $i \circ f = \id$ and $g \circ j = \id$.
\end{definition}
\begin{remark}
  The reader is warned that this definition differs from that of \cite{BGT} Definition 5.18.
\end{remark}

\begin{lemma}
  Suppose 
  \[\mathscr{A} \xrightarrow{f} \mathscr{B} \xrightarrow{g} \mathscr{C}\]
  is a split-exact sequence of presentable stable categories.
  Let $i$ and $j$ be the right adjoints of $f$ and $g$ respectively.
  Then, $j$ is fully faithful and for $\mathcal{M} \in \mathscr{B}$
  \[fi\mathcal{M} \to \mathcal{M} \to jg\mathcal{M}\]
  is exact in $\mathscr{B}$.
\end{lemma}
\begin{proof}
  $g \circ j = \id$ implies $j$ is fully faithful.
  Consider the fibre $K$ of $\mathcal{M} \to jg\mathcal{M}$. It is easy to see that it is in the kernel
  of $g$, and hence the image of $f$. Hence we have $K \cong fiK$ and therefore $K$ is the fibre of
  $fi\mathcal{M} \to fijg\mathcal{M}$. Since $ij=0$, we conclude $K \cong fi\mathcal{M}$.
\end{proof}
In the reverse direction, we have
\begin{lemma} \label{sequel}
  Suppose 
  \begin{equation} \label{seque}
    \mathscr{A} \xrightarrow{f} \mathscr{B} \xrightarrow{g} \mathscr{C}
  \end{equation}
  is a sequence of presentable stable categories which composes to zero, where $f$ and $g$ are colimit-preserving.
  Let $i$ and $j$ be the right adjoints of $f$ and $g$ respectively.
  If $i$ and $j$ also preserve colimits and $i \circ f = \id$ and $g \circ j = \id$
  and for any $\mathcal{M} \in \mathscr{B}$, the sequence 
  \[fi\mathcal{M} \to \mathcal{M} \to jg\mathcal{M}\]
  is exact in $\mathscr{B}$, then the sequence (\ref{seque}) is split-exact.
\end{lemma}
\begin{proof}
  $i \circ f = \id$ and $g \circ j = \id$
  guarantees that $f$ and $j$ are fully-faithful. It remains to check that
  \[\mathscr{B}/\mathscr{A} \to \mathscr{C}\]
  is an equivalence. Suppose
  \[H : \mathscr{B} \to \mathscr{D}\]
  is a colimit-preserving functor which vanishes on $\mathscr{A}$. Then,
  using the sequence
  \[fi\mathcal{M} \to \mathcal{M} \to jg\mathcal{M}\]
  one can easily check that
  \[H \cong Hjg\]
  and hence there is a unique functor $F :\mathscr{C} \to \mathscr{D}$,
  namely $F\cong Hj$, such that $H$ factors as the projection functor $g$ composed with
  $F$. Therefore $\mathscr{C}$ is the desired cofibre.
\end{proof}
We note that in \cite{highertrace}, this is taken as definition for a split-exact sequence.

\begin{remark} \label{tensorref}
  The above lemma is very useful as it provides a purely 2-categorical way
  to check if a sequence is split-exact. For example, it implies that tensoring a split-exact
  sequence of module categories with another module category gives another split-exact sequence.
  More generally, after defining the notion of split-exactness in a purely 2-categorical way using the
  above lemma, it will be preserved under functors between presentable 2-categories, in the language of
  \cite{Stef}. We may also sometimes refer to split-exact sequences as Bousfield localization sequences.
\end{remark}
\section{Some Facts about Closed Immersions} \label{closedimmer}
In this section, we record some algebraic geometry facts about closed immersions of spectral Noetherian schemes which we need in the text.
We adopt the terminology of \cite{SAG}.
Let $z: Z \to X$ be a closed immersion of homologically bounded, spectral Noetherian schemes.

Let $\Gamma_{Z}(\QC{X})$ be the subcategory of
$\QC{X}$ consisting of sheaves supported on $Z$
(i.e. $j^*\mathcal{F} = 0$ where $j$ is the inclusion of the complement open).
Let
\[i_{Z} : \Gamma_{Z}(\QC{X}) \to \QC{X}\]
denote the inclusion functor of that subcategory. Let
\[\Gamma_{Z} : \QC{X} \to \Gamma_{Z}(\QC{X})\]
be the right adjoint of $i_{Z}$, or equivalently the left Kan extension
of the identity functor on $\Gamma_{Z}(\QC{X})$ to the entirety of $\QC{X}$.
$i_{Z}$ preserves compact objects since $\Gamma_{Z}$ is colimit-preserving.
We think of $\Gamma_{Z}$ as the functor taking a quasicoherent sheaf on $X$
to its local cohomology with respect to $Z$. Suppose $U$ is the complement open of $Z$,
then the following is a split-exact sequence of categories (see Definition \ref{splitexact})
\begin{equation} \label{spliseq}
  \Gamma_{Z}(\QC{X}) \to \QC{X} \to \QC{U}
\end{equation}

\begin{lemma} \label{compactgen}
  For $Z$ as above, the category $\Gamma_{Z}(\QC{X})$ is compactly generated.
\end{lemma}
\begin{proof}
  By Proposition 8.2.5.1 of \cite{SAG}, we can reduce to showing the full category of connective objects
  is compactly generated (same as the reduction of Proposition 9.6.1.1 to Proposition 9.6.1.2 in \textit{loc. cit.}).
  Then, by choosing the scallop decomposition to start with a cover of the complement of $Z$, the same arguments
  (of Proposition 9.6.2.1 of \cite{SAG} which is just a rewording of Proposition 9.6.1.2) carries through completely.
\end{proof}

\begin{lemma} \label{conserve}
  For the closed immersion $z : Z \to X$
  \[\tilde{z}^*: \Gamma_Z(\QC{X}) \to \QC{Z}\]
  and
  \[\tilde{z}^{\times} : \Gamma_Z(\QC{X}) \to \QC{Z}\] are both conservative.
\end{lemma}
\begin{proof}
  Without loss of generality, let us assume that $X$ is affine and $X \cong \Spet{R}$.
  Then $Z \cong \Spet{S}$ where $\pi_0(R) \to \pi_0(S)$ is a surjective with finitely generated
  kernel.
  Let $(t_1,\ldots,t_n)$ denote a sequence (not necessarily regular) in $\pi_0(R)$ generating $I$. We let
  $R/(t_1,\ldots,t_n)$ be the (derived) $R$-module constructed from the Koszul complex on the sequence.
  Then the subcategory of $\Mod{R}$ generated under colimits by coherent (discrete) $\pi_0(S)$ modules (which is
  the same as the subcategory generated under colimits by just $\pi_0(S)$), thought of as $R$-modules,
  contains $R/(t_1,\ldots,t_n)$ because $\pi_i(R/(t_1,\ldots,t_n))$ is a finite $\pi_0(S)$ module for all $i$ and
  $R/(t_1,\ldots,t_n)$ is homologically bounded.
  Therefore, if $\tilde{z}^*M \cong 0$, then $M \otimes_R S \otimes_S \pi_0(S) \cong M \otimes_R \pi_0(S) \cong 0$ and
  hence $M \otimes R/(t_1,\ldots,t_n) \cong 0$ which shows that $M$ is zero, since it is supported
  on $Z$. The second claim similarly since the subcategory of $R$ modules generated under colimits
  from $S$ contains $\pi_0(S)$ and therefore also $R/(t_1, \ldots t_n)$.
\end{proof}
  
\begin{proposition} \label{AGinput2}
  Suppose $z : Z \to X$ is a closed immersion which factors through an étale map $u : U \to X$, then
  \[\Gamma_Z(\QC{X}) \cong \Gamma_Z(\QC{U})\]
\end{proposition}
\begin{proof}
  The pushforward functor $u_*$ induces a functor
  \[\tilde{u}_*:=\Gamma_Zu_*i_Z : \Gamma_Z(\QC{U}) \to \Gamma_Z(\QC{X})\]
  Also, the pullback functor from $\QC{X}$ to $\QC{U}$ gives a functor
  \[\tilde{u}^*:=\Gamma_Zu^*i_Z : \Gamma_Z(\QC{X}) \to \Gamma_Z(\QC{U})\]
  We wish to show that they are inverses, i.e. that $\tilde{u}^*\tilde{u}_* \cong \id$ and $\tilde{u}_*\tilde{u}^* \cong \id$.
  But, after the previous lemma, it suffices to check the equations after pulling back to $Z$. Hence, the only thing to check is that
  $z^*\tilde{u}_* \cong z_U^*$. Where $z_U : Z \to U$ is the closed immersion of $Z$ into $U$.

  Let $Z'= u^{-1}(Z)$, and $z' : Z' \to U$ be the closed immersion base-changed from $X$. Then, by base-change
  with the diagram (using Theorem \ref{BC1})
  \[
    \begin{tikzcd}
      Z'  \arrow{r}{z'} \arrow{d}{p} & U \arrow{d}{u}\\
      Z \arrow{r}{z} & X 
    \end{tikzcd}
  \]
  we have,
  \[z^*\tilde{u}_* \cong \tilde{p}_*\tilde{z}^{'*}\]
  where $p: Z' \to Z$ is the projection map and
  \[\tilde{z}^{'*} : \Gamma_Z(\QC{U}) \to \Gamma_Z(\QC{Z'}) \]
  is the pullback map, where the codomain is the quasicoherent sheaves on $Z' \cong Z \times_X U$ supported
  on the subscheme $Z$, thought of as a subscheme with the diagonal embedding $\delta$.
  Therefore, it suffices to show that $\tilde{\delta}^* \cong \tilde{p}_*$, as functors from
  $\Gamma_Z(\QC{Z'})$ to $\QC{Z}$. This can be shown after precomposing with $\tilde{\delta}_*$.
  Hence we have reduced to showing that $\delta^* \delta_* \cong \id$, which would follow from $Z \times_{Z'} Z \cong Z$
  (using the quasicoherent base-change isomorphism).
  But this is true because $Z \to Z'$ is étale as it is a section of an étale map (one can show it is formally étale after
  reducing to the case of underived schemes).
  
\end{proof}
\begin{remark}
  We note that is also the main input of the proof of
  Nisnevich descent in K-theory.
\end{remark}
\bibliography{ref}{}

\providecommand{\bysame}{\leavevmode\hbox to3em{\hrulefill}\thinspace}
\providecommand{\MR}{\relax\ifhmode\unskip\space\fi MR }
\providecommand{\MRhref}[2]{%
  \href{http://www.ams.org/mathscinet-getitem?mr=#1}{#2}
}
\providecommand{\href}[2]{#2}
\begin{thebibliography}{ATJLL14}

\bibitem[AILN10]{AILN}
Luchezar~L. Avramov, Srikanth~B. Iyengar, Joseph Lipman, and Suresh Nayak,
  \emph{Reduction of derived {H}ochschild functors over commutative algebras
  and schemes}, Adv. Math. \textbf{223} (2010), no.~2, 735--772. \MR{2565548}

\bibitem[ATJLL14]{lipman}
Leovigildo Alonso~Tarr\'{\i}o, Ana Jerem\'{\i}as~L\'{o}pez, and Joseph Lipman,
  \emph{Bivariance, {G}rothendieck duality and {H}ochschild homology, {II}:
  {T}he fundamental class of a flat scheme-map}, Adv. Math. \textbf{257}
  (2014), 365--461. \MR{3187654}

\bibitem[BDS16]{BAS}
Paul Balmer, Ivo Dell'Ambrogio, and Beren Sanders, \emph{Grothendieck-{N}eeman
  duality and the {W}irthm\"{u}ller isomorphism}, Compos. Math. \textbf{152}
  (2016), no.~8, 1740--1776. \MR{3542492}

\bibitem[Ber19]{Ber2}
Dario Beraldo, \emph{Sheaves of categories with local actions of {H}ochschild
  cochains}, Compos. Math. \textbf{155} (2019), no.~8, 1521--1567. \MR{3977319}

\bibitem[Ber21]{Ber1}
\bysame, \emph{The center of the categorified ring of differential operators},
  J. Eur. Math. Soc. (JEMS) \textbf{23} (2021), no.~6, 1999--2049. \MR{4244521}

\bibitem[BGT13]{BGT}
Andrew~J. Blumberg, David Gepner, and Gon\c{c}alo Tabuada, \emph{A universal
  characterization of higher algebraic {$K$}-theory}, Geom. Topol. \textbf{17}
  (2013), no.~2, 733--838. \MR{3070515}

\bibitem[BZFN10]{BZFN}
David Ben-Zvi, John Francis, and David Nadler, \emph{Integral transforms and
  {D}rinfeld centers in derived algebraic geometry}, J. Amer. Math. Soc.
  \textbf{23} (2010), no.~4, 909--966. \MR{2669705}

\bibitem[Gro64]{EGA4}
A.~Grothendieck, \emph{\'{E}l\'{e}ments de g\'{e}om\'{e}trie alg\'{e}brique.
  {IV}. \'{E}tude locale des sch\'{e}mas et des morphismes de sch\'{e}mas.
  {I}}, Inst. Hautes \'{E}tudes Sci. Publ. Math. (1964), no.~20, 259.
  \MR{173675}

\bibitem[HSS17]{highertrace}
Marc Hoyois, Sarah Scherotzke, and Nicol\`o Sibilla, \emph{Higher traces,
  noncommutative motives, and the categorified {C}hern character}, Adv. Math.
  \textbf{309} (2017), 97--154. \MR{3607274}

\bibitem[ILN15]{ILN}
Srikanth~B. Iyengar, Joseph Lipman, and Amnon Neeman, \emph{Relation between
  two twisted inverse image pseudofunctors in duality theory}, Compos. Math.
  \textbf{151} (2015), no.~4, 735--764. \MR{3334894}

\bibitem[Jia23]{upcoming}
Andy Jiang, \emph{The derived ring of differential operators}, 2023.

\bibitem[Khu17]{Khus}
Muhammad~Hafiz Khusyairi, \emph{Grothendieck duality for flat morphisms}, Ph.D.
  thesis, College of Physical and Mathematical Sciences, The Australian
  National University, 2017.

\bibitem[Lur09]{HTT}
Jacob Lurie, \emph{Higher topos theory}, Annals of Mathematics Studies, vol.
  170, Princeton University Press, Princeton, NJ, 2009. \MR{2522659}

\bibitem[Lur17]{HA}
\bysame, \emph{Higher algebra}, 2017, Available for download at
  https://www.math.ias.edu/~lurie/.

\bibitem[Lur18]{SAG}
\bysame, \emph{Spectral algebraic geometry}, 2018, Available for download at
  https://www.math.ias.edu/~lurie/.

\bibitem[MGS21]{MGS}
Aaron Mazel-Gee and Reuben Stern, \emph{A universal characterization of
  noncommutative motives and secondary algebraic k-theory}, 2021.

\bibitem[Nee14]{Basechangeneeman}
Amnon Neeman, \emph{An improvement on the base-change theorem and the functor
  $f^!$}, 2014.

\bibitem[Nee18]{N1}
\bysame, \emph{The relation between {G}rothendieck duality and {H}ochschild
  homology}, {$K$}-{T}heory---{P}roceedings of the {I}nternational
  {C}olloquium, {M}umbai, 2016, Hindustan Book Agency, New Delhi, 2018,
  pp.~91--126. \MR{3930045}

\bibitem[Nee20]{N2}
\bysame, \emph{Grothendieck duality made simple}, {$K$}-theory in algebra,
  analysis and topology, Contemp. Math., vol. 749, Amer. Math. Soc.,
  [Providence], RI, 2020, pp.~279--325. \MR{4087643}

\bibitem[Per19]{Perry}
Alexander Perry, \emph{Noncommutative homological projective duality}, Adv.
  Math. \textbf{350} (2019), 877--972. \MR{3948688}

\bibitem[Ste20]{Stef}
Germ{\'a}n Stefanich, \emph{Presentable $(\infty, n)$-categories}, 2020.

\bibitem[SVdB97]{SVdB}
Karen~E. Smith and Michel Van~den Bergh, \emph{Simplicity of rings of
  differential operators in prime characteristic}, Proc. London Math. Soc. (3)
  \textbf{75} (1997), no.~1, 32--62. \MR{1444312}

\end{thebibliography}
\bibliographystyle{amsalpha}
\end{document}